\documentclass[12pt,reqno]{amsart}
\usepackage{}
\usepackage{amsmath}
\usepackage{mathrsfs}
\usepackage{amssymb}
\usepackage{amsthm}
\usepackage{mathrsfs}
\usepackage[centertags]{amsmath}
\usepackage{amsfonts}
\usepackage[numbers,sort&compress]{natbib}
\usepackage{color}
\usepackage{extarrows}
\usepackage{fullpage}
\usepackage[colorlinks=true,  linkcolor=blue, citecolor=blue]{hyperref}
\input amssym.def
\input amssym.tex

\date{}

\newtheorem{Theorem}{Theorem}[section]

\newtheorem{Lemma}{Lemma}[section]
\newtheorem{Remark}{Remark}[section]

\newcommand\R{\mbox{\bf R}}
\newcommand\T{\mbox{\bf T}}

\newcommand\z{\mbox{\bf z}}
\newcommand\SR{\mbox{\scriptsize\bf R}}

\newcommand{\definition}{{\lower .5ex
  \hbox{$\>\>\stackrel{\triangle}{=}\>\>$} }}
\newcommand\supp{\mathop{\rm supp}}



\allowdisplaybreaks
\begin{document}

\baselineskip=22pt
\thispagestyle{empty}

\begin{center}
{\Large \bf  Spatial decay and nonlinear smoothing of  the  generalized  Ostrovsky equation}\\[1ex]

{\quad Xiangqian Yan\footnote{Email: yanxiangqian213@126.com}$^a$, Wei Yan\footnote{Email: 011133@htu.edu.cn}$^b$,
 Meihua Yang\footnote{Email: yangmeih@hust.edu.cn}$^{d*}$}\\[1ex]

{$^a$School of Mathematics,
 South China University of Technology,}\\
 {Guangzhou, Guangdong 510640, P. R. China}\\[1ex]

{$^b$College of Mathematics and Statistics, Henan Normal University,}\\
{Xinxiang, Henan 453007,   China}\\[1ex]

{$^c$School of Mathematics and Statistics, Huazhong University of Science and Technology, }\\
{Wuhan, Hubei 430074,  China}\\[1ex]

\end{center}

\bigskip
\bigskip

\noindent{\bf Abstract.}  This paper is devoted to studying the  generalized
Ostrovsky equation
\begin{eqnarray*}
      u_{t}-\beta\partial_{x}^{3}u-\gamma\partial_{x}^{-1}u+\frac{1}{k+1}(u^{k+1})_{x}=0,k\geq5
    \end{eqnarray*}
    with $\beta<0,\gamma>0$. Firstly, by using  the  density theorem in the mixed Lebesgue spaces,
    we prove that $X_{s,b}\hookrightarrow C(\R;H^{s}(\R))\hookrightarrow C(\R;L_{x}^{\infty})$
    with $s>1/2,b>1/2.$
    Secondly, we present a new proof of the convergence problem of linear
     Ostrovsky equation, which is slightly different from the proof of Theorem 1.1
      (Convergence problem of Ostrovsky equation with rough data and random data,
      Indiana Univ. Math. J. 71(2022), 1897-1921.)
Thirdly, we investigate the pointwise convergence problem  of the generalized
Ostrovsky equation. Fourthly, for the solution  $u$ to the Cauchy  problem for the  generalized
Ostrovsky equation, we prove that $u=u_{1}+u_{2},t\in[-\delta,\delta]$, and $u_{2}$ possesses
better regularity than $u$, where $u_{1}$ is the linear part of $u$ and $u_{2}$ is the
nonlinear integral part. Fifthly,
we investigate the nonlinear smoothing and the uniform convergence problem
of the generalized
Ostrovsky equation. Finally,  when  data $f$ belongs to $H^{s}(\R)(s>\frac{1}{2}-\frac{2}{k+1},k\geq6)$ and
$\lim\limits_{|x|\rightarrow{\infty}}f=0$ and
$\mathscr{F}_{x}(U(t)f)\in L^{1}(\R),$ for $t\in [-\delta,\delta],$
we prove that $\lim\limits_{|x|\rightarrow{\infty}}u=0$.
The key ingredients are high-low frequency technique,  maximal  function estimates related to
 low frequency and  some Strichartz estimates which
can be proved with the aid of  the Stein complex interpolation Theorem.


\medskip
\noindent {\bf Corresponding Author:} Meihua Yang

\medskip
\noindent {\bf Email Address:} yangmeih@hust.edu.cn

\bigskip
\noindent {\bf MSC2020-Mathematics Subject Classification}:  35G25
\bigskip

\leftskip 0 true cm \rightskip 0 true cm

\newpage{}

\begin{center}
{\Large \bf  Spatial decay and nonlinear smoothing of  the  generalized  Ostrovsky equation}\\[1ex]

{\quad Xiangqian Yan\footnote{Email: yanxiangqian213@126.com}$^a$, Wei Yan\footnote{Email: 011133@htu.edu.cn}$^b$,
 Meihua Yang\footnote{Email: yangmeih@hust.edu.cn}$^{d*}$}\\[1ex]

{$^a$School of Mathematics,
 South China University of Technology,}\\
 {Guangzhou, Guangdong 510640, P. R. China}\\[1ex]

{$^b$College of Mathematics and Statistics, Henan Normal University,}\\
{Xinxiang, Henan 453007,   China}\\[1ex]

{$^c$School of Mathematics and Statistics, Huazhong University of Science and Technology, }\\
{Wuhan, Hubei 430074,  China}\\[1ex]

{$^d$School of Mathematics and Statistics, Huazhong University of Science and Technology, }\\
{Wuhan, Hubei 430074,  China}\\[1ex]

\end{center}

\noindent{\bf Abstract.}   This paper is devoted to studying the  generalized
Ostrovsky equation
\begin{eqnarray*}
      u_{t}-\beta\partial_{x}^{3}u-\gamma\partial_{x}^{-1}u+\frac{1}{k+1}(u^{k+1})_{x}=0,k\geq5
    \end{eqnarray*}
    with $\beta<0,\gamma>0$. Firstly, by using  the  density theorem in the mixed Lebesgue spaces,
    we prove that $X_{s,b}\hookrightarrow C(\R;H^{s}(\R))\hookrightarrow C(\R;L_{x}^{\infty})$
    with $s>1/2,b>1/2.$
    Secondly, we present a new proof of the convergence problem of linear
     Ostrovsky equation, which is slightly different from the proof of Theorem 1.1
      (Convergence problem of Ostrovsky equation with rough data and random data,
      Indiana Univ. Math. J. 71(2022), 1897-1921.)
Thirdly, we investigate the pointwise convergence problem  of the generalized
Ostrovsky equation. Fourthly, for the solution  $u$ to the Cauchy  problem for the  generalized
Ostrovsky equation, we prove that $u=u_{1}+u_{2},t\in[-\delta,\delta]$, and $u_{2}$ possesses
better regularity than $u$, where $u_{1}$ is the linear part of $u$ and $u_{2}$ is the
nonlinear integral part. Fifthly,
we investigate the nonlinear smoothing and the uniform convergence problem
of the generalized
Ostrovsky equation. Finally,  when  data $f$ belongs to $H^{s}(\R)(s>\frac{1}{2}-\frac{2}{k+1},k\geq6)$ and
$\lim\limits_{|x|\rightarrow{\infty}}f=0$ and
$\mathscr{F}_{x}(U(t)f)\in L^{1}(\R),$ for $t\in [-\delta,\delta],$
we prove that $\lim\limits_{|x|\rightarrow{\infty}}u=0$.
The key ingredients are high-low frequency technique,  maximal  function estimates related to
 low frequency and  some Strichartz estimates which
can be proved with the aid of  the Stein complex interpolation Theorem.
\leftskip 0 true cm \rightskip 0 true cm

\newpage

\baselineskip=20pt

\bigskip
\bigskip

\tableofcontents

\section{Introduction}\label{sec1}

\setcounter{Theorem}{0} \setcounter{Lemma}{0}

\setcounter{section}{1}

\subsection{Background of pointwise convergence of linear dispersive equations}

\indent In this paper, we investigate the following generalized Ostrovsky equation
\begin{eqnarray}
&&u_{t}-\beta\partial_{x}^{3}u-\gamma \partial_{x}^{-1}u+\frac{1}{k+1}(u^{k+1})_{x}=0,\label{1.01}\\
&&u(x,0)=f(x),\label{1.02}
 \end{eqnarray}
where $\beta<0,\gamma>0$ and
 \begin{eqnarray*}
&&\partial_{x}^{-1}u(x)=\frac{1}{2}\left(\int_{-\infty}^{x}u(y)dy-\int_{x}^{\infty}u(y)dy\right).
 \end{eqnarray*}

\indent Equation \eqref{1.01} was introduced by Levandosky and Liu in \cite{SL}, and the stability
 of solitary wave solutions was investigated. Recently, Yan and Yan \cite{YY} proved that (\ref{1.01}) is locally well-posed in $H^{s}(\R)(s>\frac{1}{2}-\frac{2}{k})$
with $k\geq5,\beta\gamma<0.$ 

Certain equations of this class have a direct relation to physical systems. In particular, when $k=2$,  (\ref{1.01}) was Ostrovsky  equation
\begin{eqnarray}
      u_{t}-\beta\partial_{x}^{3}u-\gamma\partial_{x}^{-1}u+\frac{1}{2}(u^{2})_{x}=0,\gamma>0,\label{1.03}
    \end{eqnarray}
which was introduced by Ostrovsky in \cite{O}
as a model for weakly nonlinear long waves, by taking into account of the Coriolis force, to describe
 the propagation
of surface waves in the ocean in a rotating frame of reference. When $\beta>0$,  (\ref{1.03}) is the
 Ostrovsky equation with positive dispersion.
When $\beta<0$,  (\ref{1.03}) is the Ostrovsky equation with negative dispersion.

By using techniques
developed in \cite{KPV} and conservation laws, Linares and Milan\'{e}s \cite{LM} proved that the Cauchy problem for the Ostrovsky equation \eqref{1.03}
is locally well-posed in $X_{s}=\{f\in H^{s}(\R): \partial_{x}^{-1}f\in L^{2}(\R)\}$ with $s>\frac{3}{4}$,
and globally well-posed in $X_{1}$ when $\beta\cdot\gamma>0$.
By  employing the Fourier restriction norm method introduced in
 \cite{Bourgain-S,Bourgain-GAFA93}, Huo and Jia \cite{HJ}
proved local well-posedness of \eqref{1.03} in $\tilde{H}^{s}(\R)$, defined via the norm $\|u\|_{\tilde{H}^{s}(\R)}
=\|u\|_{H^{s}(\SR)}+\left\|\mathscr{F}_{x}^{-1}
\left(\frac{\mathscr{F}_{x}u(\xi)}{\xi}\right)\right\|_{H^{s}(\SR)}$
for $s\geq -\frac{1}{8}$ and globally well-posed in
  $\tilde{H}^{0}(\R)$. Here
 $\mathscr{F}_{x}, \mathscr{F}_{x}^{-1}$ denote the Fourier transform and its inverse in the spatial variable, respectively.
   Gui and Liu \cite{GL} proved that
 the Cauchy problem for \eqref{1.03} with positive dispersion
is locally well-posed in $H^{s}(\R)(s\geq -\frac{7}{12})$ and
  globally well-posed in $L^{2}(\R)$.
Isaza and Mej\'{i}a \cite{IM,IM2009} proved that the  Cauchy
  problem  for  \eqref{1.03} is
locally well-posed in $H^{s}(\R)(s>-\frac{3}{4})$, and the
 ill-posedness in $H^{s}(\R)$ for $s<-\frac{3}{4}$.
 Tsugawa \cite{T} established the local well-posedness  of  the Cauchy problem for \eqref{1.03}
   in some anisotropic Sobolev space
  $H^{s,a}(\R)$ with $s>-\frac{a}{2}-\frac{3}{4}$ and $0\leq a\leq 1$.
  By using Besov-type spaces and some useful algebraic identities, Li and
   his coauthors \cite{LHY,YLHD} proved that
   the  Cauchy  problem for  (\ref{1.03}) is locally well-posed
    in $H^{-\frac{3}{4}}(\R)$. 

Now we recall the research history of pointwise convergence problem.
Carleson \cite{Carleson} firstly investigated the pointwise convergence problem of  one
dimensional Schr\"odinger equation in $H^{s}(\R)(s\geq 1/4)$.
 Some  authors have investigated  the pointwise
 convergence problem of the  Schr\"odinger equation  in   dimensions $n\geq2$
  \cite{Bourgain1992,Bourgain2016,CLV,Cowling,DK,DG,Du,DGZ,DZ,GS2008,Lee,
LR2015,LR2017,LR2019, MVV,RVV,RV,Shao,S,Vega}.  Dahlberg and Kenig \cite{DK}
 showed that $s\geq \frac{1}{4}$ is
the necessary condition in $H^{s}(\R^{n})$ for the pointwise
convergence problem of the Schr\"odinger equation in any dimension. Sj\"olin \cite{S}
and Vega \cite{Vega} independently showed that the pointwise
 convergence problem of the  Schr\"odinger equation  holds in $H^{s}(\R^{n})(s>\frac{1}{2})$ in  all  dimensions.
The sufficient condition  for the pointwise
 convergence problem of the  Schr\"odinger equation was improved by Bourgain \cite{Bourgain1995},
  Moyua-Vargas-Vega \cite{MVV}, Tao-Vargas \cite{TV} and Lee
\cite{Lee}. In 2013, Bourgain\cite{Bourgain2013}  showed that the pointwise
 convergence problem of the  Schr\"odinger equation  holds in $H^{s}(\R^{n})(s>\frac{1}{2}-\frac{1}{4n})$.
Improved counterexamples are presented by Luc$\grave{a}$-Rogers \cite{LR2015,LR2017} and
 Demeter-Guo \cite{DG}. In 2016, Bourgain \cite{Bourgain2016}
  presented the counterexample  to show
 that the pointwise convergence problem requires  $s\geq\frac{n}{2(n+1)}$. In 2019, Luc$\grave{a}$
  and  Rogers \cite{LR2019} presented a different example  with ergodic arguments, which shows that
 the pointwise convergence problem of Schr\"odinger
 equation requires  $s\geq\frac{n}{2(n+1)}$.
Du et al. \cite{DGL} proved
 that
 the pointwise convergence problem
  of two dimensional
  Schr\"odinger
 equation in $H^{s}(\R^2)(s>\frac{1}{3})$ can hold with the aid of decoupling and polynomial partitioning.
 Du and Zhang \cite{DZ}
  proved that     the
  pointwise convergence problem of $n$ dimensional Schr\"odinger
 equation in $H^{s}(\R^n)(s>\frac{n}{2(n+1)},n\geq3)$ can hold.
 Miao et al. \cite{MYZ2015}  established an improved maximal inequality for $2D$
 fractional order Schr\"odinger
operators and Maio et al. \cite{MZZ2015}  established  the maximal estimates for Schr\"odinger
 equation with inverse-square potential. Moreover,
 some authors have studied the  pointwise
 convergence  of the   Schr\"odinger equation
  on the torus $\mathrm{\T}^{n}$ \cite{MV,CLS,EL}. Wang and Zhang \cite{WZ} studied the  pointwise
   convergence of solutions to the Schr\"odinger equation on manifolds. Recently, Yan et al. \cite{YZY}
    studied the  pointwise convergence of solutions to the
 Schr\"odinger equation in Fourier-Lebesgue space with rough data and random data.
 Kenig et al.
   have investigated the pointwise convergence problem of KdV equation in
    $H^{s}(\R)(s\geq \frac{1}{4})$ \cite{KPV1991,KPVCPAM}.
 Associated to the wave equation, Rogers and Villarroya \cite{RV} have proved that
$\frac{1}{2}\left[e^{it\sqrt{-\Delta}}+e^{-it\sqrt{-\Delta}}\right]
f\longrightarrow f$ with $f\in H^{s}(\R^{n})$
if and only if $s>\frac{n}{n+1}$. Recently,  Yan et al. \cite{YZDY} proved
 that $s\geq\frac{1}{4}$ is
the necessary and sufficient condition for the pointwise convergence problem
 of the linear Ostrovsky equation.

\subsection{Background of  pointwise convergence of nonlinear dispersive equations}
Compaan et al. \cite{CLS} studied the pointwise convergence of nonlinear
 Schr\"odinger equation
 with the aid of the Fourier restriction norm method. By using  the  dyadic
   mixed Lebesgue spaces, Linares and  Ramos \cite{LR}
studied the pointwise convergence of generalized Zakharov-Kuznetsov equation.

\subsection{Background of nonlinear smoothing and  uniform convergence of nonlinear dispersive equations}
Now we explain the nonlinear smoothing of nonlinear dispersive equations.
The nonlinear
 part of the solution in the integral  form
is  more regular  than the linear part.
For the nonlinear smoothing of nonlinear dispersive equations, we refer the readers
 to \cite{BS,Correia2021,CLS,CS2020,LS,ET-MRL,ET-IMRN,ET-APDE, YWY, YLHHY}.

\subsection{Motivation}

Tao    has  proved that   $X_{s,b}\hookrightarrow L^{\infty}(\R;H^{s}(\R))$
 with $s\in \R, b>\frac{1}{2}$
 in  Corollary 2.10 of \cite{Tao}.     It is obviously proved that $H^{s}(\R)\hookrightarrow L^{\infty}(\R)$  with $s>1/2.$

{\bf Question 1:     Whether it is possible to prove that
$X_{s,b}(\R^{2})\hookrightarrow C(\R;H^{s}(\R))\hookrightarrow C(\R; L^{\infty}(\R))$
 with $s>1/2,b>1/2$?}

By using high-low frequency technique and the maximal function estimates,
Yan et al. \cite{YZDY} proved that
\begin{eqnarray}
\lim\limits_{t\rightarrow0}U(t)f=f\label{1.04}
\end{eqnarray}
for almost everywhere $x\in \R.$

{\bf Question 2:     Whether it is possible to prove (\ref{1.04})
for almost everywhere $x\in \R $  with a  slightly different method?}

By using truncated frequency technique and Fourier restriction norm method,
Compaan et al. \cite{CLS} established the pointwise convergence of cubic Schr\"odinger equation.

{\bf Question 3:     Whether it is possible to establish
 the pointwise convergence of (\ref{1.01})-(\ref{1.02})
  with a  slightly different method?}

By using the Fourier restriction norm method,
Compaan et al. \cite{CLS} studied the nonlinear smoothing of
 cubic Schr\"odinger equation.

{\bf Question 4:     Whether it is possible to establish nonlinear
 smoothing of (\ref{1.01})-(\ref{1.02})?}

By using the Fourier restriction norm method, 
Compaan et al. \cite{CLS} studied the uniform convergence of cubic
 Schr\"odinger equation. More precisely,
Compaan proved that
\begin{eqnarray}
\lim\limits_{t\longrightarrow 0}\left\|u-e^{it \partial_{x}^{2}}u_{0}\right\|_{L_{x}^{\infty}}=0,\label{1.05}
\end{eqnarray}
where $u$ is the solution to cubic cubic Schr\"odinger equation.

{\bf Question 5:     Whether it is possible to establish uniform
 convergence of (\ref{1.01})-(\ref{1.02})?}

When data $\lim \limits_{|x|\longrightarrow\infty}u_{0}=0,$ if $\mathscr{F}_{x}u_{0}(\xi)\in L^{1}$,
then it is easily checked that  $\lim\limits_{|x|\longrightarrow +\infty}U(t)u_{0}=0,$
for the detail, we refer the readers to (\ref{9.011}).

{\bf Question 6:     Whether it is possible to prove
  $\lim \limits_{|x|\longrightarrow\infty}u=0,$  where $u$ is the
   solution to (\ref{1.01})-(\ref{1.02})?}

\subsection{Main Contents}

   In this paper,   we plan to study the Cauchy  problem for
     (\ref{1.01})-(\ref{1.02}) with $\beta<0,\gamma>0$.
Firstly, by using  the  density theorem in the mixed Lebesgue spaces,
 we prove that $X_{s,b}\hookrightarrow C(\R;H^{s}(\R))\hookrightarrow C(\R;L_{x}^{\infty})$
    with $s>1/2,b>1/2.$
Secondly, we present a new proof of the pointwise convergence problem
 of linear Ostrovsky equation.
Thirdly, by using  the idea slightly different from \cite{CLS} and
 using the high-low frequency technique,
 we investigate the pointwise convergence problem of the generalized
Ostrovsky equation (\ref{1.01})-(\ref{1.02}). Fourthly, following the idea
 of \cite{LS,CLS} and using
 the high-low frequency technique, for the solution  $u=u_{1}+u_{2},
 t\in[-\delta,\delta]$ of (\ref{1.01})-(\ref{1.02}),
  we prove that $u_{2}$ possesses  better regularity than $u$, where
  $u_{1}$ is the linear part of $u$ and $u_{2}$ is the
nonlinear integral part.  Fifthly,
we investigate the nonlinear smoothing and the uniform convergence problem
of the generalized
Ostrovsky equation. Finally,  when  data $f$ belongs
 to $H^{s}(\R)(s>\frac{1}{2}-\frac{2}{k+1})$ and
$\lim\limits_{|x|\rightarrow{\infty}}f=0$ and
$\mathscr{F}_{x}(U(t)f)\in L^{1}(\R),$ for $t\in [-\delta,\delta],$
we prove that $\lim\limits_{|x|\rightarrow{\infty}}u=0$.

The main difficulty is that the phase function possesses the zero singular point.

Without loss of generality, we assume that $\beta=-1$ and $\gamma=1$.

We give some notations before presenting the main results. Denote
 by $|A|$ the Lebesgue measure of $A.$
$0<\epsilon\ll1$ means that $\epsilon>0$ sufficiently approaches the zero.
$a\sim b$ means that there exists two positive constants $C_{1},C_{2}$
such that $C_{1}|a|\leq |b|\leq C_{2}|a|.$ We denote by $\mathscr{S}$ the rapidly
decreasing  function spaces and by $\mathscr{S}^{\prime}$ the slowly increasing function spaces.
We define $\langle\cdot\rangle=1+|\cdot|. $ Let $\Psi \in C^{\infty}(\R)$
satisfy $0\leq\Psi \leq1$,$supp\Psi\in[-2,2]$ and  $\Psi(t)=1$ for $|t|\leq 1$,     $\Psi(t)=0$ for $|t|>2.$
 We define $\Psi_{\delta}(t)=\Psi(\frac{t}{\delta}).$
\begin{align*}
&\mathscr{F}_{x}f(\xi)=\frac{1}{\sqrt{2\pi}}\int_{\SR}e^{-ix \xi}f(x)dx,
~\mathscr{F}_{x}^{-1}f(\xi)=\frac{1}{\sqrt{2\pi}}\int_{\SR}e^{ix \xi}f(x)dx,\\
&\mathscr{F}u(\xi,\tau)=\frac{1}{2\pi}\int_{\SR^{2}}e^{-ix\xi-it\tau}u(x,t)dxdt,
~\mathscr{F}^{-1}u(\xi,\tau)=\frac{1}{2\pi}\int_{\SR^{2}}e^{ix\xi+it\tau}u(x,t)dxdt,\\
&D_{x}^{\alpha}f=\frac{1}{\sqrt{2\pi}}\int_{\SR} |\xi|^{\alpha}e^{ix\xi}
\mathscr{F}_{x}f(\xi)d\xi,~J_{x}^{\alpha}f=\frac{1}{\sqrt{2\pi}}
\int_{\SR} \langle\xi\rangle^{\alpha}e^{ix\xi}
\mathscr{F}_{x}f(\xi)d\xi,\\
&U(t)f=\frac{1}{\sqrt{2\pi}}\int_{\SR} e^{ix\xi+it(\xi^{3}-\xi^{-1})}
\mathscr{F}_{x}f(\xi)d\xi,\\
&P_{ N}u=\frac{1}{\sqrt{2\pi}}\int_{|\xi|\leq N}e^{ix\xi}\mathscr{F}_{x}
u(\xi)d\xi,P^{ N}u=\frac{1}{\sqrt{2\pi}}\int_{|\xi|> N}e^{ix\xi}\mathscr{F}_{x}u(\xi)d\xi,\\
&\|f\|_{L_{xt}^{p}}=\|f\|_{L_{x}^{p}L_{t}^{p}},~\|f\|_{L_{t}^{p}L_{x}^{q}}=\left(\int_{\SR}
\left(\int_{\SR}|f(x,t)|^{q}dx\right)^{\frac{p}{q}}dt\right)^{\frac{1}{p}},\\
&H^{s}(\R)=\left\{f\in \mathscr{S}^{\prime}(\R):\|f \|_{H^{s}(\SR)}=
 \|\langle\xi\rangle ^{s}\mathscr{F}_{x}{f}\|_{L_{\xi}^{2}(\SR)}<\infty\right\}.
\end{align*}

 The space $X_{s,b}(\R^{2})$ is defined as follows:
\begin{eqnarray*}
X_{s,b}(\R^{2})=\left\{u\in \mathscr{S}^{\prime}(\R^{2}):\|u\|_{X_{s,b}}
=\left[\int_{\SR^{2}}\langle \xi\rangle^{2s}\langle \sigma\rangle^{2b}
|\mathscr{F}u(\xi,\tau)|^{2}d\xi d\tau\right]^{\frac{1}{2}}<\infty\right\}.
\end{eqnarray*}
Here, $\langle \sigma\rangle =1+|\tau-\phi(\xi)|$ and $\phi(\xi)=\xi^{3}-\xi^{-1}$.

The space $\hat{L}^{\infty}(\R)$ is defined as follows
\begin{eqnarray*}
&&\hat{L}^{\infty}(\R)=\{u\in\mathcal{S}^{\prime}(\R):
\|u(x)\|_{\hat{L}_{x}^{\infty}}=\|\mathscr{F}_{x}u(\xi)\|_{L_{\xi}^{1}}<\infty\}.
\end{eqnarray*}

\subsection{Main results}

The main results of this paper are as follows:

\begin{Theorem} \label{Theorem1}
Let  $b>\frac{1}{2},s>1/2$. Then,  we have
\begin{eqnarray}
&& X_{s,b}(\R^{2})\hookrightarrow C(\R;L_{x}^{\infty}).\label{1.06}
\end{eqnarray}
\end{Theorem}

\begin{Theorem} \label{Theorem2}
(Almost everywhere convergence of solution to the linear Ostrovsky equation) Let
  $f\in H^{s}(\R)(s\geq\frac{1}{4})$. Then,   we have
\begin{eqnarray*}
U(t)f\longrightarrow f
\end{eqnarray*}
as $t\longrightarrow 0$ for a.e. $x\in \R$, where $U(t)f$ is the solution
 to the linear Ostrovsky equation.
\end{Theorem}

\begin{Remark}  In order to deal with the singular point of the phase functions
 for linear Ostrovsky equation,
we use the high-low frequency  technique, with the aid of the density theorem \cite{Du}:
 $\forall \epsilon>0,$ $f\in H^{s}(\R)$, there exists a decreasing function
  $g$ and $h\in H^{s}(\R)$
with $\|h\|_{H^{s}}<\epsilon$. By using the triangle inequality, we have
\begin{eqnarray}
&&\lim\limits_{t\longrightarrow0}\left|U(t)f-f\right|\leq
\lim\limits_{t\longrightarrow0}\left|U(t)g-g\right|
+\lim\limits_{t\longrightarrow0}\left|U(t)h-h\right|.\nonumber
\end{eqnarray}

Note that, for $x\in \R,$ there exists a $y_{0}\in \R$ such that
$x\in [y_{0}-\frac{1}{2}, y_{0}+\frac{1}{2}]:=B_{1}.$  In this paper,
by using (\ref{2.01}) and \eqref{2.026} (Maximal function estimates related to
 $X_{s,b}$ related to low frequency)
  as well as $H^{s}(\R)\hookrightarrow L^{4}(\R)(s\geq\frac{1}{4}),$ we have
\begin{eqnarray*}
&&\left|\left\{x\in B_{1}: \lim\limits_{t\longrightarrow0}\left|U(t)h-h\right|>\frac{\alpha}{2}\right\}\right|
\nonumber\\&&\leq \left|\left\{x\in B_{1}: \sup\limits_{0<t<1}
\left|P^{8}h\right|>\frac{\alpha}{6}\right\}\right|\nonumber\\&&\qquad\qquad +
\left|\left\{x\in B_{1}: \sup\limits_{0<t<1}\left|P_{8}\Psi(t)h\right|>\frac{\alpha}{6}\right\}\right|
+\left|\left\{x\in B_{1}:\left|h\right|>\frac{\alpha}{6}\right\}\right|\nonumber\\
&&\qquad\leq \int_{\left\{x\in B_{1}: \sup\limits_{0<t<1}\left|P^{8}h\right|>\frac{\alpha}{6}\right\}}
\left|\frac{ \sup\limits_{0<t<1}\left|P^{8}h\right|}{\frac{\alpha}{6}}\right|^{4}dx+
\int_{\left\{x\in B_{1}: \sup\limits_{0<t<1}\left|P_{8}\Psi(t)h\right|>\frac{\alpha}{6}\right\}}
\left|\frac{ \sup\limits_{0<t<1}\left|P_{8}\Psi(t)h\right|}{\frac{\alpha}{6}}\right|dx\nonumber\\
&&\qquad\qquad+C\|h\|_{L_{x}^{4}}\nonumber\\
&&\leq C\left\|U(t)P^{8}h\right\|_{L_{x}^{4}L_{t}^{\infty}}+C \left|B_{1}\right|
\left\|U(t)P_{8}\Psi(t)h\right\|_{L_{xt}^{\infty}}+\|h\|_{H^{s}}\nonumber\\
&&\leq C\|h\|_{H^{s}}+C\|\Psi(t)U(t)h\|_{X_{0,b}}+C\|h\|_{H^{s}}\nonumber\\
&&\leq C\|h\|_{H^{s}}+C\|h\|_{L^{2}}+\|h\|_{H^{s}}\leq C\|h\|_{H^{s}}\leq \epsilon.
\end{eqnarray*}
Thanks to the high-low frequency estimates for decreasing function
 $g$ (see Lemmas 2.3, 2.4 in \cite{YZDY}),  we also have
\begin{eqnarray}
\left|U(t)g-g\right|\leq C\left[|t|+\epsilon+\frac{|t|}{\epsilon}\right].\label{1.07}
\end{eqnarray}
Hence, for any $\alpha>0$ (fixed), we proved that $\left|{\rm E_{\alpha}}\right|=0,$ where
\begin{eqnarray}
&&{\rm E_{\alpha}}=\left\{x\in B_{1}: \lim\limits_{t\longrightarrow0}\left|U(t)f-f\right|>\alpha\right\}.\nonumber
\end{eqnarray}
This completes the proof of Theorem 1.2.

The proof in this paper is slightly different from that in Yan et al. \cite{YZDY}.
The authors in \cite{YZDY} obtained for $\forall \alpha>0,$  \begin{eqnarray}
\left|\left\{x\in \R: \lim\limits_{t\longrightarrow0}\left|U(t)f-f\right|>\alpha\right\}\right|=0.\label{1.08}
\end{eqnarray}

From maximal function estimates for $h$ (see Lemma 2.1 in \cite{YZDY}), with
 the Sobolev embeddings Theorem, they establised
\begin{eqnarray}
\|P^{8}h\|_{L_{x}^{4}}\leq C\|h\|_{H^{\frac{1}{4}}},\label{1.09}
\end{eqnarray}
and
\begin{eqnarray}
\left\|U(t)P^{8}h\right\|_{L_{x}^{4}L_{t}^{\infty}}\leq C\|h\|_{H^{\frac{1}{4}}},
\left|U(t)P_{8}h-P_{8}h\right|\leq2\epsilon.\label{1.010}
\end{eqnarray}
With \eqref{1.07}, they obtained
\begin{eqnarray}
&&\left|\left\{x\in \R: \lim\limits_{t\longrightarrow0}\left|U(t)f-f\right|>\alpha\right\}\right|\nonumber\\&&
\leq \left|\left\{x\in \R: \lim\limits_{t\longrightarrow0}\left|U(t)
g-g\right|>\frac{\alpha}{2}\right\}\right|+\left|\left\{x\in \R: \lim\limits_{t\longrightarrow0}\left|U(t)h
-h\right|>\frac{\alpha}{2}\right\}\right|\nonumber\\
&&\leq C\alpha^{-4}\epsilon^{4}.\label{1.011}
\end{eqnarray}
\end{Remark}

\begin{Theorem} \label{Theorem3}
(Pointwise convergence of (\ref{1.01})-(\ref{1.02})) Let  $k\geq5$ and
$f\in H^{s}(\R)(s>\max\{\frac{1}{2}-\frac{2}{k},\frac{1}{4}\})$. Then,   we have
\begin{eqnarray*}
u(x,t)\longrightarrow f
\end{eqnarray*}
as $t\longrightarrow 0$ for a.e. $x\in \R$.
\end{Theorem}

\begin{Remark}Motivated by Lemma 4.1 of \cite{ZYYZ}, we prove Theorem 1.3.
The key to prove Theorem 1.3 is to present the decomposition
  of $u$ to (\ref{1.01})-(\ref{1.02}).
$\forall \epsilon_{1}>0,$  $u=\Psi(t)U(t)f-\frac{\Psi_{\delta}(t)}{k+1}\int_{0}^{t}U(t-\tau)\partial_{x}(u^{k+1})d\tau
=F_{1N}(x,t)+F_{2}^{N}(x,t),t\in[-\delta,\delta].$
Here
\begin{eqnarray*}
&&\hspace{-1cm}(i):\Psi(t)U(t)f_{N}\in C([-\delta,\delta];C^{\infty}(\R)),\\
&&\hspace{-1cm}(ii):F_{1N}(x,t)=\Psi(t)U(t)f_{N}-\frac{\Psi_{\delta}(t)}{k+1}
\int_{0}^{t}U(t-\tau)P_{N}\partial_{x}(u_{N}^{k+1})d\tau\in C([-\delta,\delta];C^{\infty}(\R)),\\
&&\hspace{-1cm}(iii):F_{2}^{N}(x,t)=\Psi(t)U(t)f^{N}
-\frac{\Psi_{\delta}(t)}{k+1}\left[\int_{0}^{t}U(t-\tau)\partial_{x}(u^{k+1})d\tau
-\int_{0}^{t}U(t-\tau)P_{N}\partial_{x}(u_{N}^{k+1})d\tau\right],
\end{eqnarray*}
where
\begin{eqnarray*}
&&u_{N}=\frac{1}{\sqrt{2\pi}}\int_{|\xi|\leq N}e^{ix\xi}\mathscr{F}_{x}u(\xi,t)d\xi,\\
&&f^{N}=\frac{1}{\sqrt{2\pi}}\int_{|\xi|\geq N}e^{ix\xi}\mathscr{F}_{x}f(\xi)d\xi,\\
&&\|F_{2}^{N}\|_{L_{x}^{4}L_{t}^{\infty}}<C\|F_{2}^{N}\|_{X_{s_{1},b}}<C\epsilon_{1}\left(s_{1}\geq\frac{1}{4}\right).
\end{eqnarray*}
For the details of the proof of Theorem 1.3, we refer the readers to Section 6 in this paper.

\end{Remark}

\begin{Theorem} \label{Theorem4}(Nonlinear smoothing)
Let  $k\geq6$ and $f\in H^{s}(\R)(s>\frac{1}{2}-\frac{2}{k+1})$. Then,
$u=u_{1}+u_{2}$, $t\in [-\delta,\delta]$, is the solution to (\ref{1.01})-(\ref{1.02}),
where $u_{1}=U(t)f\in H^{s}(\R), u_{2}(t)=\frac{1}{k+1}\int_{0}^{t}U(t-\tau)
\partial_{x}(u^{k+1})d\tau \in X_{\frac{1}{2}+\epsilon,b}\subset C(\R;H^{\frac{1}{2}+\epsilon}(\R))\subset C(\R; L_{x}^{\infty}(\R))$.
\end{Theorem}

\begin{Theorem} \label{Theorem5}(Uniform convergence of solution to the solution to linear equation)
Let  $k\geq6$ and $f\in H^{s}(\R)(s>\frac{1}{2}-\frac{2}{k+1})$. Then,   we have
\begin{eqnarray*}
\lim\limits_{t\longrightarrow 0}\left\|u(x,t)-U(t)f\right\|_{L_{x}^{\infty}}=0.
\end{eqnarray*}
\end{Theorem}
\begin{Remark}  We present the outline of proof of Theorem 1.5.
We firstly prove  that  there exists a unique solution $u\in X_{s,b}(b>\frac{1}{2})$ to (\ref{1.01})-(\ref{1.02})
 for $f\in H^{s}(\R)(s>\frac{1}{2}-\frac{2}{k+1},k\geq6)$, which is just Theorem 1.4.
Then, by using  Lemma \ref{Lemma3.7},  we have
\begin{eqnarray}
\lim\limits_{t\longrightarrow 0}\left\|u-U(t)f\right\|_{L_{x}^{\infty}}=0,\label{1.012}
\end{eqnarray}
since
\begin{eqnarray}
\left\|u(x,0)-f\right\|_{L_{x}^{\infty}}=0.\label{1.013}
\end{eqnarray}
\end{Remark}

\begin{Theorem} \label{Theorem6}(Pointwise spatial decay with respect to $t$).
Let  $k\geq6$ and $f\in H^{s}(\R)(s>\frac{1}{2}-\frac{2}{k+1})$. Then, for $t\in [-\delta,\delta],$  we have
\begin{eqnarray*}
\lim\limits_{|x|\longrightarrow +\infty}(u(x,t)-U(t)f)=0.
\end{eqnarray*}
When $f\in \hat{L}^{\infty}(\R)\cap H^{s}(\R)$,
for $\forall t\in [-\delta,\delta],$  we have that
\begin{eqnarray*}
\lim\limits_{|x|\longrightarrow+\infty}u=0.
\end{eqnarray*}

\end{Theorem}

The rest of the paper is arranged as follows.
In Section 2,  we give some preliminaries.
In Section 3, we prove Theorem 1.1. In Section 4, we show a multilinear estimate.
In Section 5, we prove Theorem 1.2. In Section 6, we prove Theorem 1.3.
In Section 7, we prove Theorem 1.4. In Section 8, we prove Theorem 1.5. In Section 9, we prove Theorem 1.6.

\bigskip
\section{Preliminaries}\label{sec2}

\setcounter{equation}{0}

\setcounter{Theorem}{0}

\setcounter{Lemma}{0}

\setcounter{section}{2}

\begin{Lemma}\label{Lemma2.1}
Let $\delta\in (0,1)$,  $s\in \R$, $c\in \R$,   $-\frac{1}{2}<b^{\prime}
\leq0\leq b\leq b^{\prime}+1$, $f\in H^{s}(\R)$ and $h\in X_{s,b^{\prime}}$.
Then,   we have
\begin{eqnarray}
&&\left\|\Psi(t)U(t)f\right\|_{X_{s,c}(\SR^{2})}\leq C\|f\|_{H^{s}(\SR)},\label{2.01}\\
&&\left\|\Psi\left(\frac{t}{\delta}\right)\int_{0}^{t}U(t-\tau)h(\tau)
d\tau\right\|_{X_{s,b}(\SR^{2})}\leq C
\delta^{1+b^{\prime}-b}\|h\|_{X_{s,b^{\prime}}(\SR^{2})}.\label{2.02}
\end{eqnarray}
\end{Lemma}

For the proof of Lemma 2.1, we refer the readers to \cite{Bourgain-GAFA93,G2002}.

\begin{Lemma}\label{Lemma2.2}(An interpolation Theorem related to $X_{s,b})$.)
\noindent Let $b=\frac{1}{2}+\frac{\epsilon}{24}$ and $0\leq s\leq \frac{1}{2}.$
Then,     we have
\begin{eqnarray}
\left\|I^{s}(u_{1},u_{2})\right\|
  _{L_{xt}^{2}}\leq C\prod_{j=1}^{2}\| u_{j}\| _{X_{0,\frac{2+2s}{3}b}},\label{2.03}
\end{eqnarray}
where $I^{s}(u_{1},u_{2})$ is defined as follows:
\begin{eqnarray}\label{2.04}
\hspace{-1cm}\mathscr{F}I^{s}(u_{1},u_{2})(\xi,\tau)
&=&\int_{\!\!\!\mbox{\scriptsize $
\begin{array}{l}
\xi=\xi_{1}+\xi_{2}\\
\tau=\tau_{1}+\tau_{2}
\end{array}
$}}
\left|\phi^{\prime}(\xi_{1})-\phi^{\prime}(\xi_{2})\right|^{s}
\mathscr{F}{u_{1}}(\xi_{1},\tau_{1})\mathscr{F}{u_{2}}
(\xi_{2},\tau_{2})\,d\xi_{1}d\tau_{1}.
\end{eqnarray}

\end{Lemma}

For the proof of Lemma 2.2, we refer the readers to \cite{G2002}.

\begin{Lemma}\label{Lemma2.3}(Strichartz estimates related to $X_{s,b}.$)
Let $8\leq q<\infty,N\geq2, 0<M\leq2$,  $s=\frac{1}{8}-\frac{1}{q}$, $s_{1}>\frac{1}{4}$, $b=\frac{1}{2}+\frac{\epsilon}{24}$
 and $0<\epsilon\leq 10^{-3}$. Then, we have
\begin{align}
&\|u\|_{L_{xt}^{q}}\leq C\|u\|_{X_{4s,b}},\label{2.05}\\
&\|D_{x}^{\frac{1}{6}}P^{ N}u\|_{L_{xt}^{6}}\leq C\|u\|_{X_{0,b}},\label{2.06}\\
&\|u\|_{L_{xt}^{8}}\leq C\|u\|_{X_{0,b}}\label{2.07},\\
&\|u\|_{L_{xt}^{\frac{8}{1+\epsilon}}}\leq C\|u\|_{X_{0,\frac{3-\epsilon}{3}b}}
\leq C\|u\|_{X_{0,\frac{1}{2}-\frac{\epsilon}{12}}}\label{2.08},\\
&\|\partial_{x}P^{ N}u\|_{L_{x}^{\infty}L_{t}^{2}}\leq C\|u\|_{X_{0,b}},\label{2.09}\\
&\|D_{x}^{s_{1}}\Psi(t)P_{ M}u\|_{L_{x}^{2}L_{t}^{\infty}}\leq C\|u\|_{X_{0,b}},\label{2.010}\\
&\|D_{x}^{1-2\epsilon}P^{ N}u\|_{L_{x}^{\frac{1}{\epsilon}}L_{t}^{2}}\leq C
\|u\|_{X_{0,(1-2\epsilon)b}}\label{2.011}.
\end{align}
\end{Lemma}
\noindent{\bf Proof.} For the proof of Lemma 2.3,  we refer the readers to Lemma 2.2 of  \cite{YY}.

\begin{Lemma}\label{Lemma2.4}(Strichartz estimates related to $X_{s,b}.$)
Let $N\geq2$ and $b=\frac{1}{2}+\frac{\epsilon}{24}$
 and $0<\epsilon\leq 10^{-3}$. Then, we have
\begin{align}
&\|D_{x}^{\frac{2-\epsilon}{12}}P^{ N}u\|_{L_{xt}^{\frac{6}{1+\epsilon}}}
\leq C\|u\|_{X_{0,\frac{2-\epsilon}{2}b}}\leq C\|u\|_{X_{0,\frac{1}{2}-\frac{\epsilon}{12}}}\label{2.012},\\
&\left\|D_{x}^{\frac{1}{6}-\epsilon}P^{N}u\right\|_{L_{xt}^{\frac{12}{2-3\epsilon}}}\leq C
\|u\|_{X_{0,b}}\label{2.013}.
\end{align}
\end{Lemma}
\noindent{\bf Proof.} Define
\begin{eqnarray}
G_{z}u:=C\int_{\SR^{2}}e^{ix\xi+it\tau}\left(\frac{|\xi|}{\langle\sigma\rangle^{6b}}\right)^{\frac{z}{6}}\mathscr{F}P^{N}u(\xi,\tau)d\xi d\tau\label{2.014}.
\end{eqnarray}
Obviously, $G_{z}$ is analytic on $S:=\left\{z|z=x+iy,x\in(0,1)\right\}$ and continuous on $\bar{S}:=\left\{z|z=x+iy,x\in[0,1]\right\}$.
By using the Plancherel identity, we have
\begin{eqnarray}
\left\|G_{iy}u\right\|_{L_{xt}^{2}}=C\left\|\left(\frac{|\xi|}{\langle\sigma\rangle^{6b}}\right)^{\frac{iy}{6}}
\mathscr{F}P^{N}u(\xi,\tau)\right\|_{L_{\xi\tau}^{2}}\leq C\|u\|_{L_{xt}^{2}}\label{2.015}.
\end{eqnarray}
From (\ref{2.06}), we have
\begin{eqnarray}
\left\|G_{1+iy}u\right\|_{L_{xt}^{6}}\leq C\|u\|_{L_{xt}^{2}}\label{2.016}.
\end{eqnarray}
By using Stein complex interpolation and (\ref{2.015})-(\ref{2.016}),
we have
\begin{eqnarray}
\left\|G_{\frac{2-\epsilon}{2}}u\right\|_{L_{xt}^{\frac{6}{1+\epsilon}}}\leq C\|u\|_{L_{xt}^{2}}\label{2.017}.
\end{eqnarray}
From (\ref{2.017}), we have that (\ref{2.012}) is valid.

Define
\begin{eqnarray}
T_{z}u:=C\int_{\SR^{2}}e^{ix\xi+it\tau}|\xi|^{\frac{z}{6}}\langle\sigma\rangle^{-b}\mathscr{F}P^{N}u(\xi,\tau)d\xi d\tau\label{2.018}.
\end{eqnarray}
Obviously, $T_{z}$ is analytic on $S:=\left\{z|z=x+iy,x\in(0,1)\right\}$ and continuous on $\bar{S}:=\left\{z|z=x+iy,x\in[0,1]\right\}$.
From (\ref{2.07}), we have
\begin{eqnarray}
\left\|T_{iy}u\right\|_{L_{xt}^{8}}\leq C\|u\|_{L_{xt}^{2}}\label{2.019}.
\end{eqnarray}
From (\ref{2.06}), we have
\begin{eqnarray}
\left\|T_{1+iy}u\right\|_{L_{xt}^{6}}\leq C\|u\|_{L_{xt}^{2}}\label{2.020}.
\end{eqnarray}
By using Stein complex interpolation  and (\ref{2.019})-(\ref{2.020}),
we have
\begin{eqnarray}
\left\|T_{1-6\epsilon}u\right\|_{L_{xt}^{\frac{12}{2-3\epsilon}}}\leq C\|u\|_{L_{xt}^{2}}\label{2.021}.
\end{eqnarray}
(\ref{2.021}) yields that (\ref{2.013}) is valid.

We completed the proof of Lemma 2.4.

\begin{Lemma}\label{Lemma2.5}(Maximal function estimates related to $X_{s,b}$ with high frequency.)
Let $f\in L_{x}^{2}(\R)$ and $\gamma\geq 4$ and $\supp\mathscr{F}_{x}f\subseteq[N,4N]\cup[-4N,-N]$, where $N\in2^{\z}$, $N\geq 1$. Then,  we have
\begin{eqnarray}
&&\left\|U(t)f(x)\right\|_{L_{x}^{\gamma}L_{t}^{\infty}}\leq CN^{\frac{1}{2}-\frac{1}{\gamma}}\left\|f\right\|_{L_{x}^{2}},\label{2.022}
\end{eqnarray}
where
\begin{eqnarray*}
&&U(t)f(x)=\int_{\SR}e^{ix\xi+it\phi(\xi)}\mathscr{F}_{x}f(\xi)\chi_{[N,4N]}(|\xi|)d\xi.
\end{eqnarray*}
Let $b>\frac{1}{2}$, $\gamma\geq4$ and $\supp\mathscr{F}_{x}u\subseteq[N,4N]\cup[-4N,-N]$, where $N\in2^{\z},N\geq1$. Then, we have
\begin{eqnarray}
&&\left\|u\right\|_{L_{x}^{\gamma}L_{t}^{\infty}}\leq CN^{\frac{1}{2}-\frac{1}{\gamma}}\left\|u\right\|_{X_{0,b}}.\label{2.023}
\end{eqnarray}
Let $b>\frac{1}{2}$, $\gamma\geq4$ and $\supp\mathscr{F}_{x}u\subseteq[1,\infty)\cup(-\infty,-1]$. Then,  we have
\begin{eqnarray}
&&\left\|u\right\|_{L_{x}^{\gamma}L_{t}^{\infty}}\leq C\left\|u\right\|_{X_{s,b}}.\label{2.024}
\end{eqnarray}
Here $s=\frac{1}{2}-\frac{1}{\gamma}+\epsilon$.

\end{Lemma}

For the proof of Lemma 2.5,  we refer the readers to Lemma 2.9 of \cite{YY}.

\begin{Lemma}\label{Lemma2.6}(Maximal function estimate related to $X_{s,b}$ with high frequency.)
Let $u\in X_{0,b}(b>\frac{1}{2})$ and $\supp\mathscr{F}_{x}u\subseteq[N,4N]\cup[-4N,-N](N\geq1)$.   Then,  we have
\begin{eqnarray}
&&\left\|u\right\|_{L_{xt}^{\infty}}\leq CN^{\frac{1}{2}}\left\|u\right\|_{X_{0,b}}.\label{2.025}
\end{eqnarray}

\end{Lemma}

For the proof of Lemma 2.6,  we refer the readers to (2.31) of \cite{YYDH}.

\begin{Lemma}\label{Lemma2.7}(Maximal function estimates related to $X_{s,b}$ with low frequency.)
Let $\supp\mathscr{F}_{x}u\subset (0,2]\cup [-2,0)$ and $s_{1}\in (\frac{1}{4},\frac{1}{2})$. Then, we have
\begin{eqnarray}
&&\left\|P_{2}\Psi(t)u\right\|_{L_{xt}^{\infty}}
\leq C\left\|u\right\|_{X_{0,b}},\label{2.026}\\
&&\|D_{x}^{(1-2\epsilon)s_{1}}P_{2}\Psi(t)u\|_{L_{x}^{\frac{2}{1-2\epsilon}}L_{t}^{\infty}}\leq C
\|u\|_{X_{0,b}}.\label{2.027}
\end{eqnarray}
\end{Lemma}
\noindent{\bf Proof.} We only consider  $\supp\mathscr{F}_{x}u\subset (0,2]$ since the case $\supp\mathscr{F}_{x}u\subset  [-2,0)$
can be obtained similarly.
By using the Cauchy-Schwarz inequality and the Plancherel identity,
we have
\begin{eqnarray}
&&\|P_{2}\Psi(t)u\|_{L_{x}^{\infty}}=\frac{1}{\sqrt{2\pi}}
\left\|\sum\limits_{k=-\infty}^{0}\int_{2^{k-1}}^{2^{k+1}}
e^{ix\xi}\Psi(t)\mathscr{F}_{x}u(\xi,t)d\xi\right\|_{L_{x}^{\infty}}\nonumber\\
&&\leq C\sum\limits_{k=-\infty}^{0}\int_{2^{k-1}}^{2^{k+1}}
|\Psi(t)\mathscr{F}_{x}u(\xi,t)|d\xi\nonumber\\
&&\leq C\sum\limits_{k=-\infty}^{0}2^{\frac{k}{2}}\left[\int_{2^{k-1}}^{2^{k+1}}
|\Psi(t)\mathscr{F}_{x}u(\xi,t)|^{2}d\xi\right]^{\frac{1}{2}}\nonumber\\
&&\leq  C\sum\limits_{k=-\infty}^{0}2^{\frac{k}{2}}\left\|\mathscr{F}_{x}^{-1}
\left(\chi_{[2^{k-1},2^{k+1}]}(\xi)\Psi(t)\mathscr{F}_{x}u(\xi,t)\right)\right\|_{L_{x}^{2}}.
\label{2.028}
\end{eqnarray}
By using (\ref{2.028}), Minkowski's inequality and (\ref{2.010}),   we have
\begin{eqnarray}
&&\|P_{2}\Psi(t)u\|_{L_{x}^{\infty}L_{t}^{\infty}}\leq  C\sum\limits_{k=-\infty}^{0}2^{\frac{k}{2}}\left\|\mathscr{F}_{x}^{-1}
\left(\chi_{[2^{k-1},2^{k+1}]}(\xi)\Psi(t)\mathscr{F}_{x}u(\xi,t)\right)\right\|_{L_{t}^{\infty}L_{x}^{2}}\nonumber\\
&&\leq C\sum\limits_{k=-\infty}^{0}2^{\frac{k}{2}}\left\|\mathscr{F}_{x}^{-1}
\left(\chi_{[2^{k-1},2^{k+1}]}(\xi)\Psi(t)\mathscr{F}_{x}u(\xi,t)\right)\right\|_{L_{x}^{2}L_{t}^{\infty}}\nonumber\\
&&\leq C\sum\limits_{k=-\infty}^{0}2^{\frac{k}{2}-s_{1}k}\left\|D_{x}^{s_{1}}
\mathscr{F}_{x}^{-1}\left(\chi_{[2^{k-1},2^{k+1}]}(\xi)\Psi(t)\mathscr{F}_{x}u(\xi,t)\right)\right\|_{L_{x}^{2}L_{t}^{\infty}}\nonumber\\
&&\leq C\sum\limits_{k=-\infty}^{0}2^{\frac{k}{2}-s_{1}k}\left\|\mathscr{F}_{x}^{-1}
\left(\chi_{[2^{k-1},2^{k+1}]}(\xi)\Psi(t)\mathscr{F}_{x}u(\xi,t)\right)\right\|_{X_{0,b}}\nonumber\\
&&\leq C\left[\sum\limits_{k=-\infty}^{0}2^{k-2s_{1}k}\right]^{\frac{1}{2}}\left[\sum\limits_{k=-\infty}^{0}
\left\|\mathscr{F}_{x}^{-1}\left(\chi_{[2^{k-1},2^{k+1}]}(\xi)\Psi(t)\mathscr{F}_{x}u(\xi,t)\right)\right\|_{X_{0,b}}^{2}\right]^{\frac{1}{2}}
\nonumber\\
&&\leq C\|u\|_{X_{0,b}}.\label{2.029}
\end{eqnarray}
Here $\frac{1}{4}<s_{1}<\frac{1}{2}.$
Interpolating (\ref{2.026}) with (\ref{2.010}) yields  (\ref{2.027}).

We completed the proof of Lemma 2.7.

\noindent{\bf Remark 6:} In fact,  if we replace   $\supp\mathscr{F}_{x}u\subset
 (0,2]\cup [-2,0)$ with $\supp\mathscr{F}_{x}u\subset (0,M]\cup [-M,0)$ for arbitrary $M>0,$
Lemma 2.7 is also true.

\begin{Lemma}\label{Lemma2.8}(Maximal function estimates related to $X_{s,b}$.)
Let $s>\frac{1}{2},b>\frac{1}{2}$. Then, we have
\begin{eqnarray}
&&\left\|u\right\|_{L_{xt}^{\infty}}
\leq C\left\|u\right\|_{X_{s,b}}.\label{2.030}
\end{eqnarray}
\end{Lemma}
\noindent{\bf Proof.}By using a direct computation, for $s>\frac{1}{2},$ we have that
\begin{eqnarray}
&&\|U(t)f\|_{L_{xt}^{\infty}}\leq C\left\|\int_{\SR}
e^{ix\xi+it(\xi^{3}-\frac{1}{\xi})}\mathscr {F}_{x} fd\xi\right\|_{L_{xt}^{\infty}}\nonumber\\
&&\leq C\int_{\SR}\left|\mathscr {F}_{x} f\right|d\xi\leq C\|f\|_{H^{s}}.\label{2.031}
\end{eqnarray}
Obviously,
\begin{eqnarray}
u(x,t)=C\int_{\SR^{2}}e^{ix\xi+it\tau}\mathscr{F}u(\xi,\tau)d\xi d\tau=
C\int_{\SR^{2}}e^{ix\xi+it\phi(\xi)+it(\tau-\phi(\xi))}\mathscr{F}u(\xi,\tau)d\xi d\tau,\label{2.032}
\end{eqnarray}
let  $\tau-\phi(\xi)=\lambda$, from (\ref{2.032}), we have
\begin{eqnarray}
&&u(x,t)=
C\int_{\SR^{2}}e^{ix\xi+it\phi(\xi)+it\lambda}\mathscr{F}u(\xi,\phi+\lambda)d\xi d\lambda\nonumber\\
&&=C\int_{\SR}e^{it\lambda}U(t)Gd\lambda,\label{2.033}
\end{eqnarray}
where $\mathscr{F}_{x}G=\mathscr{F}u(\xi,\phi+\lambda).$ It follows from \eqref{2.033} that
\begin{eqnarray}
&&\left\|u\right\|_{L_{xt}^{\infty}}\leq C\int_{\SR}\|U(t)G\|_{L_{xt}^{\infty}}d\lambda\leq C\int_{\SR}\|G\|_{H^{s}}d\lambda\nonumber\\
&&\leq C\left[\int_{\SR}\langle \lambda\rangle^{2b}\|G\|_{H^{s}}^{2}d\lambda\right]^{1/2}\left[\int_{\SR}\langle \lambda\rangle^{-2b}d\lambda\right]^{1/2}\nonumber\\
&&\leq C\left[\int_{\SR}\langle \lambda\rangle^{2b}\|G\|_{H^{s}}^{2}d\lambda\right]^{1/2}.\label{2.034}
\end{eqnarray}
Let  $\tau-\phi(\xi)=\lambda$,  from (\ref{2.034}), we have
\begin{eqnarray}
&&\left\|u\right\|_{L_{xt}^{\infty}}\leq C\|u\|_{X_{s,b}}.\label{2.035}
\end{eqnarray}

This completes the proof of Lemma 2.8.

\bigskip

\section{ Proof of Theorem 1.1}
\setcounter{equation}{0}
\setcounter{Theorem}{0}

\setcounter{Lemma}{0}

\setcounter{section}{3}

This section is devoted to proving Theorem 1.1.

\begin{Lemma}\label{Lemma3.1}(The density Theorem in the mixed Lebesgue spaces)
Suppose that  $0< p,q<\infty$ and
\begin{eqnarray}
\left\|f\right\|_{L_{t\in I}^{q}L_{x}^{p}}:=\left[\int_{I}
\left(\int_{\SR^{d}}|f|^{p}dx\right)^{\frac{q}{p}}dt\right]^{1/q}<\infty, I\subset \R.\label{3.01}
\end{eqnarray}
Then, the functions of the form
\begin{eqnarray*}
&&\sum\limits_{i=1}^{r}a_{i}(x)\chi_{E_{i}}(t)
\end{eqnarray*}
are dense in $L_{t\in I}^{q}L_{x}^{p}$. Here
$a_{i}(x)\in L^{p}$ and $E_{i}$ are pairwise disjoint measurable sets in $I$ with finite measures
and $r\geq1$ is an integer.
\end{Lemma}

For the proof of Lemma 3.1, see Proposition 5.5.6 of \cite{G2009}.

\begin{Lemma}\label{Lemma3.2}(The triangle inequality in the mixed Lebesgue spaces)
Suppose that $1\leq  p,q<\infty$
and  $f,g\in L_{t\in I}^{q}L_{x}^{p}(\R^{d})$. Then, we have
\begin{eqnarray}
\left\|f+g\right\|_{L_{t\in I}^{q}L_{x}^{p}}\leq \left\|f\right\|_{L_{t\in I}^{q}L_{x}^{p}}
+\left\|g\right\|_{L_{t\in I}^{q}L_{x}^{p}}.\label{3.02}
\end{eqnarray}

\end{Lemma}

For the proof of Lemma 3.2, we refer the reader to Lemma A.1  of \cite{Frank}.

\begin{Lemma}\label{Lemma3.3}(The convergence  in the mixed Lebesgue spaces)
Suppose that $\left\|\langle\xi\rangle^{s}\mathscr{F}v(\xi,\tau)\right\|_{L_{\xi}^{2}L_{\tau}^{1}}<\infty$.
 Then,    for  $\epsilon>0,$ there exist  $\delta>0(<\epsilon^{4}\leq 1),M>0$ and $C>0$ such that
\begin{eqnarray}
&&\left\|\chi_{\{\xi:|\xi|\leq\delta\}}(\xi)\langle\xi\rangle^{s}\mathscr{F}v(\xi,\tau)\right\|_{L_{\xi}^{2}L_{\tau}^{1}}\leq C\epsilon,\label{3.03}\\
&&\left\|\chi_{\{\xi:|\xi|\geq M\}}(\xi)\langle\xi\rangle^{s}\mathscr{F}v(\xi,\tau)\right\|_{L_{\xi}^{2}L_{\tau}^{1}}\leq C\epsilon\label{3.04}.
\end{eqnarray}
\end{Lemma}
\noindent{\bf Proof.} Since $\left\|\langle\xi\rangle^{s}\mathscr{F}v(\xi,\tau)\right\|_{L_{\xi}^{2}L_{\tau}^{1}}<\infty$, from Lemma \ref{Lemma3.1},
we have that for $\forall \epsilon>0$ there exists
\begin{eqnarray*}
&&\sum\limits_{i=1}^{r}a_{i}(\tau)\chi_{E_{i}}(\xi)
\end{eqnarray*}
such that
\begin{eqnarray}
\left\|\sum\limits_{i=1}^{r}a_{i}(\tau)\chi_{E_{i}}(\xi)-\langle\xi\rangle^{s}\mathscr{F}v\right\|_{L_{\xi}^{2}L_{\tau}^{1}}\leq \epsilon.\label{3.05}
\end{eqnarray}
Here
$a_{i}(\tau)\in L^{1}$ and $E_{i}$ are pairwise disjoint measurable sets in $\R$ with finite measures
and $r\geq1$ is an integer.
Thus, by using Lemma \ref{Lemma3.1}, we have
\begin{eqnarray}
&&\left\|\chi_{\{\xi:|\xi|\leq\delta\}}(\xi)\langle\xi\rangle^{s}\mathscr{F}v\right\|_{L_{\xi}^{2}L_{\tau}^{1}}\nonumber\\&&\leq
\left\|\chi_{\{\xi:|\xi|\leq\delta\}}(\xi)\left(\langle\xi\rangle^{s}
\mathscr{F}v-\sum\limits_{i=1}^{r}a_{i}(\tau)\chi_{E_{i}}(\xi)\right)\right\|_{L_{\xi}^{2}L_{\tau}^{1}}+
\left\|\chi_{\{\xi:|\xi|\leq\delta\}}(\xi)\sum\limits_{i=1}^{r}a_{i}(\tau)\chi_{E_{i}}(\xi)\right\|_{L_{\xi}^{2}L_{\tau}^{1}}\nonumber\\
&&\leq \left\|\left(\langle\xi\rangle^{s}
\mathscr{F}v-\sum\limits_{i=1}^{r}a_{i}(\tau)\chi_{E_{i}}(\xi)\right)\right\|_{L_{\xi}^{2}L_{\tau}^{1}}+
\left\|\chi_{\{\xi:|\xi|\leq\delta\}}(\xi)\sum\limits_{i=1}^{r}a_{i}(\tau)\chi_{E_{i}}(\xi)\right\|_{L_{\xi}^{2}L_{\tau}^{1}}\nonumber\\
&&\leq \epsilon+\sum\limits_{i=1}^{r}\left\|\chi_{\{\xi:|\xi|\leq\delta\}}(\xi)a_{i}(\tau)\chi_{E_{i}}(\xi)\right\|_{L_{\xi}^{2}L_{\tau}^{1}}\nonumber\\
&&\leq \epsilon+Cr\delta^{1/2}\leq \epsilon+Cr\epsilon\label{3.06}
\end{eqnarray}
and
\begin{eqnarray}
&&\left\|\chi_{\{\xi:|\xi|\geq M\}}(\xi)\langle\xi\rangle^{s}\mathscr{F}v\right\|_{L_{\xi}^{2}L_{\tau}^{1}}\nonumber\\&&\leq
\left\|\chi_{\{\xi:|\xi|\geq M\}}(\xi)\left(\langle\xi\rangle^{s}
\mathscr{F}v-\sum\limits_{i=1}^{r}a_{i}(\tau)\chi_{E_{i}}(\xi)\right)\right\|_{L_{\xi}^{2}L_{\tau}^{1}}+
\left\|\chi_{\{\xi:|\xi|\geq M\}}(\xi)\sum\limits_{i=1}^{r}a_{i}(\tau)\chi_{E_{i}}(\xi)\right\|_{L_{\xi}^{2}L_{\tau}^{1}}\nonumber\\
&&\leq \left\|\left(\langle\xi\rangle^{s}
\mathscr{F}v-\sum\limits_{i=1}^{r}a_{i}(\tau)\chi_{E_{i}}(\xi)\right)\right\|_{L_{\xi}^{2}L_{\tau}^{1}}+
\left\|\chi_{\{\xi:|\xi|\geq M\}}(\xi)\sum\limits_{i=1}^{r}a_{i}(\tau)\chi_{E_{i}}(\xi)\right\|_{L_{\xi}^{2}L_{\tau}^{1}}\nonumber\\
&&\leq \epsilon+\sum\limits_{i=1}^{r}\left\|\chi_{\{\xi:|\xi|\geq M\}}(\xi)a_{i}(\tau)\chi_{E_{i}}(\xi)\right\|_{L_{\xi}^{2}L_{\tau}^{1}}\nonumber\\
&&\leq \epsilon+Cr{\rm mes}(E_{i}\cap \left\{\xi:|\xi|\geq M\right\})\leq C\epsilon\label{3.07}
\end{eqnarray}
for sufficiently large $M>0.$

This completes the proof of Lemma 3.3.

\begin{Lemma}\label{Lemma3.4}
Let  $b\in \R,s\in \R$ and
\begin{eqnarray}
\left[\int_{\SR}\int_{\SR}\langle \xi\rangle^{2s}
\langle \tau\rangle^{2b}|\mathscr{F}v(\xi,\tau)|^{2}d\xi d\tau\right]^{1/2}<\infty
.\label{3.08}
\end{eqnarray}
Then, for $\epsilon>0$, there exists $\delta >0(\leq\epsilon^{4}\leq1)$ such that
\begin{eqnarray}
\left[\int_{|\xi|<\delta}\int_{\SR}\langle \xi\rangle^{2s}
\langle \tau\rangle^{2b}|\mathscr{F}v(\xi,\tau)|^{2}d\tau d\xi \right]^{1/2}\leq C\epsilon
.\label{3.09}
\end{eqnarray}
\end{Lemma}
\noindent{\bf Proof.}
Since
\begin{eqnarray}
\langle \xi\rangle^{s}
\langle \tau\rangle^{b}|\mathscr{F}v(\xi,\tau)|\in L^{2},\label{3.010}
\end{eqnarray}
by using the density Theorem in $L^{2}(\R^{2})$, for $\epsilon>0$, there exists $g \in C_{c}^{\infty}(\R^{2})$ such that
\begin{eqnarray}
\left\|\langle \xi\rangle^{s}
\langle \tau\rangle^{b}|\mathscr{F}v(\xi,\tau)|-g\right\|_{L^{2}}\leq \epsilon,\label{3.011}
\end{eqnarray}
by using (\ref{3.011}), for $\delta\leq \epsilon^{4},$ we have
\begin{eqnarray}
&&\left\|\chi_{\{\xi:|\xi|\leq \delta\}}(\xi)\langle \xi\rangle^{s}
\langle \tau\rangle^{b}|\mathscr{F}v(\xi,\tau)|\right\|_{L^{2}}\nonumber\\
&&\leq \left\|\chi_{\{\xi:|\xi|\leq \delta\}}(\xi)\left[\langle \xi\rangle^{s}
\langle \tau\rangle^{b}|\mathscr{F}v(\xi,\tau)|-g\right]\right\|_{L^{2}}
+\left\|\chi_{\{\xi:|\xi|\leq \delta\}}(\xi)g\right\|_{L^{2}}\nonumber\\
&&\leq \epsilon+C\delta^{1/2}\leq C\epsilon
.\label{3.012}
\end{eqnarray}

This completes the proof of Lemma 3.4.

\begin{Remark}
We prove Lemma 3.4  as an independent interest.
\end{Remark}

\begin{Lemma}\label{Lemma3.5}
Let  $b>\frac{1}{2}$, we have
\begin{eqnarray*}
&&\|v\|_{L_{t}^{\infty}H^{s}}\leq C\|J_{t}^{b}J_{x}^{s}v\|_{L_{xt}^{2}}.
\end{eqnarray*}
\end{Lemma}
\noindent{\bf Proof.} 
Since $b>\frac{1}{2}$, by using the Cauchy-Schwarz inequality with respect to $\tau$, we have
\begin{eqnarray*}
&&\|v\|_{H^{s}}=\left(\int_{\SR}\langle\xi\rangle^{2s}\left|\int_{\SR}e^{it\tau}\mathscr{F}v(\xi,\tau)d\tau\right|^{2}d\xi
\right)^{\frac{1}{2}}\\
&&\leq \left(\int_{\SR}\langle\xi\rangle^{2s}\int_{\SR}\langle\tau\rangle^{2b}|\mathscr{F}v(\xi,\tau)|^{2}d\tau d\xi\right)^{\frac{1}{2}}
\left(\int_{\SR}\langle\tau\rangle^{-2b}d\tau\right)^{\frac{1}{2}}\\
&&\leq C\left(\int_{\SR}\langle\xi\rangle^{2s}\int_{\SR}\langle\tau\rangle^{2b}|\mathscr{F}v(\xi,\tau)|^{2}d\tau d\xi\right)^{\frac{1}{2}}\\
&&=C\|J_{t}^{b}J_{x}^{s}v\|_{L_{xt}^{2}}.
\end{eqnarray*}

This completes the proof of Lemma 3.5.

\begin{Lemma}\label{Lemma3.6}
Let  $b>\frac{1}{2}$, we have
\begin{eqnarray}
&& X_{s,b}(\R^{2})\hookrightarrow C(\R;H_{x}^{s}).\label{3.013}
\end{eqnarray}
\end{Lemma}
\noindent{\bf Proof.} When $u\in X_{s,b}(\R^{2})$,    by using the fact that
\begin{eqnarray}
\|u\|_{X_{s,b}}=\|J_{t}^{b}J_{x}^{s}U(-t)u\|_{L_{xt}^{2}},\label{3.014}
\end{eqnarray}
In Lemma 2.9 and   Corollary 2.10 of \cite{Tao},  Tao  has prove that
\begin{eqnarray}
\sup\limits_{t\in \SR}\|u\|_{H^{s}}\leq C\|J_{t}^{b}J_{x}^{s}U(-t)u\|_{L_{xt}^{2}}
=\|u\|_{X_{s,b}}.
\label{3.015}
\end{eqnarray}
Consequently, it suffices to prove the continuity of $\|u\|_{H^{s}}$ with respect to $t$.
  More specifically,
it remains to show that for  $\forall\epsilon>0$, there exists $\delta>0$, such that when $|t_{1}-t_{2}|<\delta$, we have
\begin{eqnarray}
&&\Big|\left\|u(t_{1})\right\|_{H^{2}}-\left\|u(t_{2})\right\|_{H^{s}}\Big|<\epsilon.\label{3.016}
\end{eqnarray}
We define
\begin{eqnarray}
&&v(t):=U(-t)u.\label{3.017}
\end{eqnarray}
Then, by using \eqref{3.014} and \eqref{3.017}, we have
\begin{eqnarray}
&&\|u\|_{X_{s,b}}=\left\|J_{x}^{s}J_{t}^{b}v\right\|_{L_{xt}^{2}}.\label{3.018}
\end{eqnarray}
By using Lemma 3.5, \eqref{3.014}, \eqref{3.018} and the fact that $u\in X_{s,b}(\R^{2})$, we have
\begin{eqnarray}
&&\|v\|_{L_{t}^{\infty}H^{s}}\leq C
\left\|J_{x}^{s}J_{t}^{b}v\right\|_{L_{xt}^{2}}=C\|u\|_{X_{s,b}}.\label{3.019}
\end{eqnarray}
From $u\in X_{s,b}(\R^{2})$, we have that
 for any $\epsilon>0$, there exists $M\geq 2026$, such that
\begin{eqnarray}
&&\int_{|\tau|\geq M}\langle\tau\rangle^{2b}
\int_{\SR}\langle\xi\rangle^{2s}|\mathscr{F}_{xt}v(\tau,\xi)|^{2}d\xi d\tau\leq \epsilon^{2}.\label{3.020}
\end{eqnarray}
From \eqref{3.020}, by using the H\"older inequality,   we have that
\begin{eqnarray}
&&\int_{|\tau|\geq M}\left(\int_{\SR}\langle\xi\rangle^{2s}|\mathscr{F}v(\tau,\xi)|^{2}d\xi\right)^{1/2}d\tau\nonumber\\
&&\leq \left[\int_{|\tau|\geq M}\langle\tau\rangle^{2b}\int_{\SR}\langle\xi\rangle^{2s}|\mathscr{F}v(\tau,\xi)|^{2}d\xi d\tau\right]^{1/2}
\times \left[\int_{|\tau|\geq M}\langle\tau\rangle^{-2b}d\tau\right]^{1/2}\nonumber\\
&&\leq C\epsilon.\label{3.021}
\end{eqnarray}
It follows from the triangle inequality that
\begin{eqnarray}
&&\Big|\|u(t_{1})\|_{H^{s}}-\|u(t_{2})\|_{H^{s}}\Big|\nonumber\\&&=\Big|\|U(t_{1})v(t_{1})\|_{H^{2}}-\|U(t_{2})v(t_{2})\|_{H^{s}}\Big|
\leq \|U(t_{1})v(t_{1})-U(t_{2})v(t_{2})\|_{H^{s}}\nonumber\\
&&\leq \|(U(t_{1})-U(t_{2}))v(t_{1})\|_{H^{s}}+\|U(t_{2})(v(t_{1})-v(t_{2}))\|_{H^{s}}\nonumber\\
&&\leq \|(U(t_{1})-U(t_{2}))v(t_{1})\|_{H^{s}}+\|v(t_{1})-v(t_{2})\|_{H^{s}}\nonumber\\
&&=I_{1}+I_{2},\label{3.022}
\end{eqnarray}
where
\begin{eqnarray*}
&&I_{1}=\left\|\left(U(t_{1})-U(t_{2})\right)v(t_{1})\right\|_{H^{s}},\, I_{2}=\|v(t_{1})-v(t_{2})\|_{H^{s}}.
\end{eqnarray*}
For $I_{1}$, note that for $M\geq2026$, we have
\begin{eqnarray}
&&\left\|\left(U(t_{1})-U(t_{2})\right)v(t_{1})\right\|_{H^{s}}^{2}=\int_{\SR}\langle\xi\rangle^{2s}
\left|\left(e^{it_{1}\phi(\xi)}-e^{it_{2}\phi(\xi)}\right)
\mathscr{F}_{x}v(\xi,t_{1})\right|^{2}d\xi\nonumber\\
&&=\int_{0\leq |\xi|\leq \epsilon^{4}}\langle\xi\rangle^{2s}
\left|\left(e^{it_{1}\phi(\xi)}-e^{it_{2}\phi(\xi)}\right)
\mathscr{F}_{x}v(\xi,t_{1})\right|^{2}d\xi\nonumber\\
&&\quad+\int_{\epsilon^{4} \leq |\xi|\leq 1}\langle\xi\rangle^{2s}
\left|\left(e^{it_{1}\phi(\xi)}-e^{it_{2}\phi(\xi)}\right)
\mathscr{F}_{x}v(\xi,t_{1})\right|^{2}d\xi\nonumber\\
&&\quad+\int_{1 \leq |\xi|\leq M}\langle\xi\rangle^{2s}
\left|\left(e^{it_{1}\phi(\xi)}-e^{it_{2}\phi(\xi)}\right)
\mathscr{F}_{x}v(\xi,t_{1})\right|^{2}d\xi\nonumber\\
&&\quad+\int_{|\xi|\geq M}\langle\xi\rangle^{2s}
\left|\left(e^{it_{1}\phi(\xi)}-e^{it_{2}\phi(\xi)}\right)
\mathscr{F}_{x}v(\xi,t_{1})\right|^{2}d\xi\nonumber\\
&&=I_{11}+I_{12}+I_{13}+I_{14},\label{3.023}
\end{eqnarray}
where
\begin{eqnarray*}
&&I_{11}=\int_{0\leq |\xi|\leq \epsilon^{4}}\langle\xi\rangle^{2s}
\left|\left(e^{it_{1}\phi(\xi)}-e^{it_{2}\phi(\xi)}\right)
\mathscr{F}_{x}v(\xi,t_{1})\right|^{2}d\xi,\\
&&I_{12}=\int_{\epsilon^{4} \leq |\xi|\leq 1}\langle\xi\rangle^{2s}
\left|\left(e^{it_{1}\phi(\xi)}-e^{it_{2}\phi(\xi)}\right)
\mathscr{F}_{x}v(\xi,t_{1})\right|^{2}d\xi,\\
&&I_{13}=\int_{1 \leq |\xi|\leq M}\langle\xi\rangle^{2s}
\left|\left(e^{it_{1}\phi(\xi)}-e^{it_{2}\phi(\xi)}\right)
\mathscr{F}_{x}v(\xi,t_{1})\right|^{2}d\xi,\\
&&I_{14}=\int_{|\xi|\geq M}\langle\xi\rangle^{2s}
\left|\left(e^{it_{1}\phi(\xi)}-e^{it_{2}\phi(\xi)}\right)
\mathscr{F}_{x}v(\xi,t_{1})\right|^{2}d\xi.
\end{eqnarray*}
For $I_{11}$, by using \eqref{3.03},   we have
\begin{eqnarray}
&&I_{11}=\int_{0\leq |\xi|\leq \epsilon^{4}}\langle\xi\rangle^{2s}
\left|\left(e^{it_{1}\phi(\xi)}-e^{it_{2}\phi(\xi)}\right)\right|^{2}
\left|\int_{\SR}e^{it_{1}\tau}\mathscr{F}v(\xi,\tau)d\tau\right|^{2}d\xi\nonumber\\
&&\leq 4\int_{0\leq |\xi|\leq \epsilon^{4}}\langle\xi\rangle^{2s}\left\|
\mathscr{F}v\right\|_{L_{\tau}^{1}}^{2}d\xi\nonumber\\
&&\leq C\epsilon^{2}.\label{3.024}
\end{eqnarray}
For $I_{12}$,  by using the Cauchy-Schawarz inequality with respect to $\tau,$ we have
\begin{eqnarray}
&&I_{12}=\int_{\epsilon^{4} \leq |\xi|\leq 1}\langle\xi\rangle^{2s}
\left|\left(e^{it_{1}\phi(\xi)}-e^{it_{2}\phi(\xi)}\right)\right|^{2}
\left|\int_{\SR}e^{it_{1}\tau}\mathscr{F}v(\xi,\tau)\right|^{2}d\xi\nonumber\\
&&\leq C\epsilon^{-8}|t_{1}-t_{2}|^{2}\int_{\epsilon^{4}\leq |\xi|\leq 1}\langle\xi\rangle^{2s}
\|\mathscr{F}v\|_{L_{\tau}^{1}}^{2}d\xi\nonumber\\
&&\leq C\epsilon^{-8}|t_{1}-t_{2}|^{2}\left[\int_{\epsilon^{4}\leq |\xi|\leq 1}\langle\xi\rangle^{2s}
\|\langle\tau\rangle^{b}\mathscr{F}v\|_{L_{\tau}^{2}}^{2}d\xi\right]\int_{\SR}\langle\tau\rangle^{-2b}d\tau\nonumber\\
&&\leq C\epsilon^{-8}|t_{1}-t_{2}|^{2}\left\|J_{x}^{s}J_{t}^{b}v\right\|_{L_{xt}^{2}}^{2}\nonumber\\
&&\leq C_{1}\epsilon^{-8}|t_{1}-t_{2}|^{2}.\label{3.025}
\end{eqnarray}
For $I_{13}$,  by using the Cauchy-Schawarz inequality with respect to $\tau,$  we have
\begin{eqnarray}
&&I_{13}=\int_{1 \leq |\xi|\leq M}\langle\xi\rangle^{2s}
\left|\left(e^{it_{1}\phi(\xi)}-e^{it_{2}\phi(\xi)}\right)\right|^{2}
\left(\int_{\SR}e^{it_{1}\tau}\mathscr{F}v(\xi,\tau)d\tau\right)^{2}d\xi\nonumber\\
&&\leq C M^{6}|t_{1}-t_{2}|^{2}\int_{1\leq |\xi|\leq M}\langle\xi\rangle^{2s}\left\|
\mathscr{F}v\right\|_{L_{\tau}^{1}}^{2}d\xi\nonumber\\
&&\leq CM^{6}|t_{1}-t_{2}|^{2}\left[\int_{1\leq |\xi|\leq M}\langle\xi\rangle^{2s}
\|\langle\tau\rangle^{b}\mathscr{F}v\|_{L_{\tau}^{2}}^{2}d\xi\right]\int_{\SR}\langle\tau\rangle^{-2b}d\tau\nonumber\\
&&\leq CM^{6}|t_{1}-t_{2}|^{2}\left\|J_{x}^{s}J_{t}^{b}v\right\|_{L_{xt}^{2}}^{2}\leq CM^{6}|t_{1}-t_{2}|^{2}.\label{3.026}
\end{eqnarray}
For $I_{14}$, by using (\ref{3.04}),   we have
\begin{eqnarray}
&&I_{14}=\int_{|\xi|\geq M}\langle\xi\rangle^{2s}
\left|\left(e^{it_{1}\phi(\xi)}-e^{it_{2}\phi(\xi)}\right)\right|^{2}
\left|\int_{\SR}e^{it_{1}\tau}\mathscr{F}v(\xi,\tau)d\tau\right|^{2}d\xi\nonumber\\
&&\leq 4\int_{|\xi|\geq M}\langle\xi\rangle^{2s}\left\|
\mathscr{F}v\right\|_{L_{\tau}^{1}}^{2}d\xi\leq C\epsilon^{2}.\label{3.027}
\end{eqnarray}
By using \eqref{3.024}-\eqref{3.027}, we have
\begin{eqnarray}
&&I_{1}\leq  C_{2}(\epsilon^{-4}|t_{1}-t_{2}|+M^{3}|t_{1}-t_{2}|+\epsilon).\label{3.028}
\end{eqnarray}
For $I_{2}$,  note that for $M\geq2026$, by Minkowski's inequality, we have
\begin{eqnarray}
&&\|v(t_{1})-v(t_{2})\|_{H^{s}}=\left\|\int_{\SR}
\left(e^{it_{1}\tau}-e^{it_{2}\tau}\right)
\mathscr{F}_{t}v(x,\tau)d\tau\right\|_{H^{s}}\nonumber\\
&&=\left(\int_{\SR}\langle\xi\rangle^{2s}\left|\int_{\SR}
\left(e^{it_{1}\tau}-e^{it_{2}\tau}\right)
\mathscr{F}v(\xi,\tau)d\tau\right|^{2}d\xi\right)^{\frac{1}{2}}\nonumber\\
&&\leq \left(\int_{\SR}\langle\xi\rangle^{2s}\left(\int_{\SR}
\left|\left(e^{it_{1}\tau}-e^{it_{2}\tau}\right)
\mathscr{F}v(\xi,\tau)\right|d\tau\right)^{2}d\xi\right)^{\frac{1}{2}}\nonumber\\
&&\leq \int_{\SR}\left|e^{it_{1}\tau}-e^{it_{2}\tau}\right|
\left(\int_{\SR}\langle\xi\rangle^{2s}
\left|\mathscr{F}v(\xi,\tau)\right|^{2}d\xi\right)^{\frac{1}{2}} d\tau\nonumber\\
&&=\int_{|\tau|\leq M}\left|e^{it_{1}\tau}-e^{it_{2}\tau}\right|
\left(\int_{\SR}\langle\xi\rangle^{2s}
\left|\mathscr{F}v(\xi,\tau)\right|^{2}d\xi\right)^{\frac{1}{2}} d\tau\nonumber\\
&&\quad+\int_{|\tau|\geq M}\left|e^{it_{1}\tau}-e^{it_{2}\tau}\right|
\left(\int_{\SR}\langle\xi\rangle^{2s}
\left|\mathscr{F}v(\xi,\tau)\right|^{2}d\xi\right)^{\frac{1}{2}} d\tau\nonumber\\
&&=I_{21}+I_{22},\label{3.029}
\end{eqnarray}
where
\begin{eqnarray*}
&&I_{21}=\int_{|\tau|\leq M}\left|e^{it_{1}\tau}-e^{it_{2}\tau}\right|
\left(\int_{\SR}\langle\xi\rangle^{2s}
\left|\mathscr{F}v(\xi,\tau)\right|^{2}d\xi\right)^{\frac{1}{2}} d\tau,\\
&&I_{22}=\int_{|\tau|\geq M}\left|e^{it_{1}\tau}-e^{it_{2}\tau}\right|
\left(\int_{\SR}\langle\xi\rangle^{2s}
\left|\mathscr{F}v(\xi,\tau)\right|^{2}d\xi\right)^{\frac{1}{2}} d\tau.
\end{eqnarray*}
For $I_{21}$,  we have
\begin{eqnarray}
&&I_{21}=\int_{|\tau|\leq M}\left|e^{it_{1}\tau}-e^{it_{2}\tau}\right|\left(\int_{\SR}\langle\xi\rangle^{2s}
\left|\mathscr{F}v(\xi,\tau)\right|^{2}d\xi\right)^{\frac{1}{2}} d\tau\nonumber\\
&&\leq M|t_{1}-t_{2}|\int_{|\tau|\leq M}\left(\int_{\SR}\langle\xi\rangle^{2s}
\left|\mathscr{F}v(\xi,\tau)\right|^{2}d\xi\right)^{\frac{1}{2}} d\tau\nonumber\\
&&\leq M|t_{1}-t_{2}|\left\|J_{x}^{s}J_{t}^{b}v\right\|_{L_{xt}^{2}}\nonumber\\
&&\leq CM|t_{1}-t_{2}|.\label{3.030}
\end{eqnarray}
For $I_{22}$, from \eqref{3.021}, we have
\begin{eqnarray}
&&I_{22}=\int_{|\tau|\geq M}\left|e^{it_{1}\tau}-e^{it_{2}\tau}\right|\left(\int_{\SR}\langle\xi\rangle^{2s}
\left|\mathscr{F}v(\xi,\tau)\right|^{2}d\xi\right)^{\frac{1}{2}} d\tau.\nonumber\\
&&\leq 4\int_{|\tau|\geq M}\left(\int_{\SR}\langle\xi\rangle^{2s}
\left|\mathscr{F}v(\xi,\tau)\right|^{2}d\xi\right)^{\frac{1}{2}} d\tau
\leq C\epsilon.\label{3.031}
\end{eqnarray}
From \eqref{3.030}-\eqref{3.031}, we have
\begin{eqnarray}
&&I_{2}\leq C(M|t_{1}-t_{2}|+\epsilon).\label{3.032}
\end{eqnarray}
From \eqref{3.028} and \eqref{3.032}, we have
\begin{eqnarray}
&&I_{1}+I_{2}\leq C_{3}(M|t_{1}-t_{2}|+\epsilon^{-4}|t_{1}-t_{2}|+M^{3}|t_{1}-t_{2}|+\epsilon).\label{3.033}
\end{eqnarray}
When  $|t_{1}-t_{2}|<\delta<{\rm min}\left\{\frac{\epsilon}{M^{3}},\epsilon^{5}\right\}$, then, we have
$I_{1}+I_{2}\leq C\epsilon,$ thus, we have
\begin{eqnarray}
&&\Big|\|u(t_{1})\|_{H^{s}}-\|u(t_{2})\|_{H^{s}}\Big|=
\Big|\|U(t_{1})v(t_{1})\|_{H^{s}}-\|U(t_{2})v(t_{2})\|_{H^{s}}\Big|<\epsilon.\label{3.034}
\end{eqnarray}

This completes the proof of Lemma 3.6.

To  prove  Theorem 1.1, it suffices to prove Lemma 3.7.

\begin{Lemma}\label{Lemma3.7}
Let  $b>\frac{1}{2},s>1/2$, we have
\begin{eqnarray}
&& X_{s,b}(\R^{2})\hookrightarrow C(\R;L_{x}^{\infty}).\label{3.035}
\end{eqnarray}
\end{Lemma}
\noindent{\bf Proof.} By using $H^{\frac{1}{2}+\epsilon}(\R)\hookrightarrow L^{\infty}(\R)$
 and Lemma \ref{Lemma3.5}, for $\forall \epsilon>0$,  $\exists \delta>0,$
when $|t_{1}-t_{2}|<\delta,$ by using  the triangle  inequality,  we have that
\begin{eqnarray}
&&\left|\|u(t_{1})\|_{L_{x}^{\infty}}-\|u(t_{2})\|_{L_{x}^{\infty}}\right|\nonumber\\
&&\leq \|u(t_{1})-u(t_{2})\|_{L_{x}^{\infty}}\leq C\|u(t_{1})-u(t_{2})\|_{H_{x}^{s}}<\epsilon.\label{3.036}
\end{eqnarray}

This completes the proof of Lemma 3.7.

\bigskip

\section{ Multilinear estimate}

\setcounter{equation}{0}

 \setcounter{Theorem}{0}

\setcounter{Lemma}{0}

 \setcounter{section}{4}
 This section  is  devoted to proving Lemma 4.1, which is used to establishing Theorem 1.4.

\begin{Lemma}\label{Lemma4.1}
Assume that  $s\geq\frac{1}{2}-\frac{2}{k+1}+88\epsilon,k\geq6$,
$s_{0}=\frac{1}{2}+\epsilon,$ $s_{1}=\frac{6-40\epsilon}{12+\epsilon},$ $b_{0}=-\frac{1}{2}+\frac{\epsilon}{12},$
and $b_{1}=\frac{1}{2}+\frac{\epsilon}{24}$.
Then, we derive
\begin{eqnarray}
\left\|\partial_{x}\left(\prod\limits_{j=1}^{k+1}\Psi(t)u_{j}\right)
\right\|_{X_{s_{0},b_{0}}}
\leq C\prod_{j=1}^{k+1}\|u_{j}\|_{X_{s,b_{1}}}.\label{4.01}
\end{eqnarray}
\end{Lemma}
\noindent {\bf Proof.} To show (\ref{4.01}),
by duality,  it  is sufficient  to show
\begin{eqnarray}
\left|\int_{\SR^{2}}J^{s_{0}}\partial_{x}
\left(\prod_{j=1}^{k+1}\Psi(t)u_{j}\right)\bar{h}dxdt
\right|\leq C\|h\|_{X_{0,-b_{0}}}
\prod_{j=1}^{k+1}\|u_{j}\|_{X_{s,b_{1}}}\label{4.02}.
\end{eqnarray}
Define
\begin{eqnarray*}
g(\xi,\tau):=\langle\sigma\rangle^{b_{0}} \mathscr{F}h(\xi,\tau),
f_{j}(\xi_{j},\tau_{j}):=\langle\xi_{j}\rangle^{s}\langle\sigma_{j}\rangle^{b_{1}}
\mathscr{F}\Psi(t)u_{j}(\xi_{j},\tau_{j})(1\leq j\leq k+1).
\end{eqnarray*}
Based on  the Plancherel identity, to show (\ref{4.02}), it  is sufficient to show
\begin{eqnarray}
\int_{\xi=\sum\limits_{j=1}^{k+1}\xi_{j}}\int_{\tau=\sum\limits_{j=1}^{k+1}\tau_{j}}
\frac{|\xi|\langle\xi\rangle^{s_{0}}g\prod\limits_{j=1}^{k+1}f_{j}}
{\langle\sigma\rangle^{-b_{0}}\prod\limits_{j=1}^{k+1}
\langle\xi_{j}\rangle^{s}\langle\sigma_{j}\rangle^{b}}d\delta
&&\leq C\|g\|_{L^{2}}\prod\limits_{j=1}^{k+1}\|f_{j}\|_{L^{2}}.\label{4.03}
\end{eqnarray}
Here $d\delta=d\xi_{1}d\xi_{2}\cdot \cdot \cdot d\xi_{k}
d\xi d\tau_{1}d\tau_{2}\cdot \cdot \cdot d\tau_{k}d\tau.$

\noindent Now we introduce some definitions:
\begin{eqnarray*}
&&K(\xi_{1},\xi_{2},\cdot\cdot\cdot,\xi_{k},\xi,\tau_{1},\tau_{2},\cdot\cdot\cdot,\tau_{k},\tau)
=:\frac{|\xi|\langle\xi\rangle^{s_{0}}}
{\langle\sigma\rangle^{-b_{0}}\prod\limits_{j=1}^{k+1}
\langle\xi_{j}\rangle^{s}\langle\sigma_{j}\rangle^{b_{1}}},\nonumber\\&&
\mathscr{F}F:=\frac{g}{\langle\sigma\rangle^{-b_{0}}},
\mathscr{F}F_{j}:=\frac{f_{j}}{\langle\sigma_{j}\rangle^{b_{1}}}
(1\leq j\leq k+1),\\
&&I_{1}:=\int_{\xi=\sum\limits_{j=1}^{k+1}\xi_{j}}
\int_{\tau=\sum\limits_{j=1}^{k+1}\tau_{j}}K(\xi_{1},\xi_{2},
\cdot\cdot\cdot,\xi_{k},\xi,\tau_{1},\tau_{2},\cdot\cdot\cdot,\tau_{k},\tau)
g\prod\limits_{j=1}^{k+1}f_{j}d\delta.
\end{eqnarray*}
By using the symmetry, we can give the assumption that $|\xi_{1}|\geq
 |\xi_{2}|\geq\cdot\cdot\cdot\geq |\xi_{k+1}|.$

\noindent By using a direct computation, we have
\begin{eqnarray*}
\Omega:=\left\{(\xi_{1},\xi_{2},\cdot\cdot\cdot,\xi_{k},\xi,\tau_{1},\tau_{2},
\cdot\cdot\cdot,\tau_{k},\tau)\in \R^{2(k+1)}:|\xi_{1}|\geq |\xi_{2}|\geq
\cdot\cdot\cdot\geq |\xi_{k+1}|\right\}\subset \bigcup\limits_{j=0}^{k+1}\Omega_{j}.
\end{eqnarray*}
Here,
\begin{eqnarray*}
&&\Omega_{0}=\left\{(\xi_{1},\xi_{2},\cdot\cdot\cdot,\xi_{k},\xi,\tau_{1},\tau_{2},
\cdot\cdot\cdot,\tau_{k},\tau)\in \Omega,|\xi_{1}|\leq80(k+1)\right\},\\
&&\Omega_{1}=\left\{(\xi_{1},\xi_{2},\cdot\cdot\cdot,\xi_{k},\xi,\tau_{1},\tau_{2},
\cdot\cdot\cdot,\tau_{k},\tau)\in \Omega_{0}^{c},|\xi_{1}|\geq80(k+1)|\xi_{2}|\right\},\\
&&\Omega_{2}=\left\{(\xi_{1},\xi_{2},\cdot\cdot\cdot,\xi_{k},\xi,\tau_{1},\tau_{2},
\cdot\cdot\cdot,\tau_{k},\tau)\in \Omega_{0}^{c}, |\xi_{1}|\sim |\xi_{2}|\geq 80(k+1)|\xi_{3}|\right\},\\
&&\Omega_{3}=\left\{(\xi_{1},\xi_{2},\cdot\cdot\cdot,\xi_{k},\xi,\tau_{1},\tau_{2},
\cdot\cdot\cdot,\tau_{k},\tau)\in \Omega_{0}^{c}, |\xi_{1}|\sim
 |\xi_{3}|\geq 80(k+1)|\xi_{4}|\right\},\\
&&\Omega_{4}=\left\{(\xi_{1},\xi_{2},\cdot\cdot\cdot,\xi_{k},\xi,\tau_{1},\tau_{2},
\cdot\cdot\cdot,\tau_{k},\tau)\in \Omega_{0}^{c},|\xi_{1}|\geq80(k+1), |\xi_{1}|\sim|\xi_{4}|\geq 80(k+1)|\xi_{5}|\right\},\\
&&\Omega_{5}=\left\{(\xi_{1},\xi_{2},\cdot\cdot\cdot,\xi_{k},\xi,\tau_{1},\tau_{2},
\cdot\cdot\cdot,\tau_{k},\tau)\in \Omega_{0}^{c}, |\xi_{1}|\sim  |\xi_{5}|\geq80(k+1)|\xi_{6}|\right\},\\
&&\Omega_{6}=\left\{(\xi_{1},\xi_{2},\cdot\cdot\cdot,\xi_{k},\xi,\tau_{1},\tau_{2},
\cdot\cdot\cdot,\tau_{k},\tau)\in \Omega_{0}^{c},|\xi_{1}|\sim |\xi_{l-1}|\geq 80(k+1)|\xi_{l}|(7\leq l\leq k+1)\right\},\\
&&\Omega_{7}=\left\{(\xi_{1},\xi_{2},\cdot\cdot\cdot,\xi_{k},\xi,\tau_{1},\tau_{2},
\cdot\cdot\cdot,\tau_{k},\tau)\in \Omega_{0}^{c}, |\xi_{1}|\sim  |\xi_{k+1}|\right\}.
\end{eqnarray*}
{\bf Case(1)}: $(\xi_{1},\xi_{2},\cdot\cdot\cdot,\xi_{k},\xi,\tau_{1},\tau_{2},\cdot\cdot\cdot,
\tau_{k},\tau)\in \Omega_{0}$. This case can be proved similarly to Case (1) of Lemma 3.1 of \cite{YY} with the aid of (\ref{2.05}).

\noindent{\bf Case(2)}: $(\xi_{1},\xi_{2},\cdot\cdot\cdot,\xi_{k},\xi,\tau_{1},\tau_{2},\cdot\cdot\cdot,\tau_{k},\tau)\in \Omega_{1}$, we consider
\begin{eqnarray}
&&|1-\frac{1}{3\xi_{1}^{2}\xi_{2}^{2}}|\geq\frac{1}{2},\label{4.04}\\
&&|1-\frac{1}{3\xi_{1}^{2}\xi_{2}^{2}}|<\frac{1}{2}.\label{4.05}
\end{eqnarray}
If (\ref{4.04}) is valid, we consider
\begin{eqnarray}
&&\left|1-\frac{1}{3\xi^{2}\xi_{3}^{2}}\right|\geq\frac{1}{2},\label{4.004}\\
&&\left|1-\frac{1}{3\xi^{2}\xi_{3}^{2}}\right|<\frac{1}{2}.\label{4.005}
\end{eqnarray}
If (\ref{4.004}) is valid,
we have
\begin{eqnarray}
&&K(\xi_{1},\xi_{2},\cdot\cdot\cdot,\xi_{k},\xi,\tau_{1},\tau_{2},\cdot\cdot\cdot,\tau_{k},\tau)\leq C
\frac{|\xi_{1}|^{\frac{2}{k+1}+\epsilon}|\xi_{1}^{2}-\xi_{2}^{2}|^{\frac{1}{2}}}
{\langle\sigma\rangle^{-b_{0}}\left(\prod\limits_{j=2}^{k+1}
\langle\xi_{j}\rangle^{s}\right)\left(\prod\limits_{j=1}^{k+1}\langle\sigma_{j}\rangle^{b_{1}}\right)}\nonumber\\
&&\leq C\frac{|\xi_{1}^{2}-\xi_{2}^{2}|^{\frac{1}{2}}
|\xi^{2}-\xi_{3}^{2}|^{s_1}|\xi_{1}|^{-\frac{12-80\epsilon}{12+\epsilon}+\frac{2}{k+1}+\epsilon}}
{\langle\sigma\rangle^{-b_{0}}\left(\prod\limits_{j=2}^{k+1}
\langle\xi_{j}\rangle^{s}\right)\left(\prod\limits_{j=1}^{k+1}\langle\sigma_{j}\rangle^{b_{1}}\right)}\nonumber\\
&&\leq C\frac{|\xi_{1}^{2}-\xi_{2}^{2}|^{\frac{1}{2}}|\xi^{2}-\xi_{3}^{2}|^{s_{1}}}
{\langle\sigma\rangle^{-b_{0}}\left(\prod\limits_{j=4}^{k+1}
\langle\xi_{j}\rangle^{\frac{ks}{k-2}+\frac{k-1}{(k+1)(k-2)}-\epsilon}\right)
\left(\prod\limits_{j=1}^{k+1}\langle\sigma_{j}\rangle^{b_{1}}\right)}
.\label{4.006}
\end{eqnarray}
By using  Lemmas \ref{Lemma2.2}, \ref{Lemma2.8} and (\ref{4.006}),
 since $s\geq\frac{1}{2}-\frac{2}{k+1}+88\epsilon,k\geq6$,
we have
\begin{eqnarray*}
&&I_{1}\leq C\left\|I^{1/2}(F_{1},F_{2})\right\|
  _{L_{xt}^{2}}\left\|I^{s_{1}}(F,F_{3})\right\|
  _{L_{xt}^{2}}\prod\limits_{j=4}^{k+1}\left\|J^{-\frac{ks}{k-2}-\frac{k-1}{(k+1)(k-2)}-\epsilon}F_{j}\right\|_{L_{xt}^{\infty}}
\nonumber\\&& \leq C\left(\prod_{j=1}^{k+1}\|f_{j}\|_{L_{\xi\tau}^{2}}\right)\|g\|_{L_{\xi\tau}^{2}}.
\end{eqnarray*}
If (\ref{4.005}) is valid,
we have $|\xi|\sim|\xi_{3}|^{-1}$, and
\begin{eqnarray}
&&K(\xi_{1},\xi_{2},\cdot\cdot\cdot,\xi_{k},\xi,\tau_{1},\tau_{2},\cdot\cdot\cdot,\tau_{k},\tau)\leq C
\frac{|\xi_{1}|^{\frac{2}{k+1}+\epsilon}|\xi_{1}^{2}-\xi_{2}^{2}|^{\frac{1}{2}}}
{\langle\sigma\rangle^{-b_{0}}\left(\prod\limits_{j=2}^{k+1}
\langle\xi_{j}\rangle^{s}\right)\left(\prod\limits_{j=1}^{k+1}\langle\sigma_{j}\rangle^{b_{1}}\right)}\nonumber\\
&&\leq C\frac{|\xi_{1}^{2}-\xi_{2}^{2}|^{\frac{1}{2}}|\xi|^{1-2\epsilon}|\xi_{3}|^{\frac{1-2\epsilon}{3}}}
{\langle\sigma\rangle^{-b_{0}}\left(\prod\limits_{j=1}^{k+1}\langle\sigma_{j}\rangle^{b_{1}}\right)}
.\label{4.007}
\end{eqnarray}
By using  Lemmas \ref{Lemma2.2}, \ref{Lemma2.8} and (\ref{4.007}), 
since $s\geq\frac{1}{2}-\frac{2}{k+1}+88\epsilon,k\geq6$,
we have
\begin{eqnarray*}
&&I_{1}\leq C\left\|I^{1/2}(F_{1},F_{2})\right\|
  _{L_{xt}^{2}}\left\|D^{1-2\epsilon}F\right\|
  _{L_{x}^{\frac{1}{\epsilon}}L_{t}^{2}}\left\|D^{\frac{1-2\epsilon}{3}}F_{3}\right\|
  _{L_{x}^{\frac{2}{1-2\epsilon}}L_{t}^{\infty}}\prod\limits_{j=4}^{k+1}\left\|F_{j}\right\|_{L_{xt}^{\infty}}
\nonumber\\&& \leq C\left(\prod_{j=1}^{k+1}\|f_{j}\|_{L_{\xi\tau}^{2}}\right)\|g\|_{L_{\xi\tau}^{2}}.
\end{eqnarray*}
When (\ref{4.05})  is valid,  we have $|\xi_{1}|\sim|\xi_{2}|^{-1}$. In this case,  we consider
\begin{eqnarray}
&&\left|1-\frac{1}{3\xi^{2}\xi_{k+1}^{2}}\right|\geq\frac{1}{2},\label{4.07}\\
&&\left|1-\frac{1}{3\xi^{2}\xi_{k+1}^{2}}\right|<\frac{1}{2},\label{4.08}
\end{eqnarray}
respectively.

\noindent
When (\ref{4.07})  is valid,    we have
\begin{eqnarray}
&&K(\xi_{1},\xi_{2},\cdot\cdot\cdot,\xi_{k},\xi,\tau_{1},\tau_{2},
\cdot\cdot\cdot,\tau_{k},\tau)\leq C
\frac{|\xi_{1}||\xi_{2}|^{\frac{1}{4}+\epsilon}|\xi|^{1-4\epsilon}}
{\langle\sigma\rangle^{-b_{0}}
\prod\limits_{j=1}^{k+1}\langle\sigma_{j}\rangle^{b_{1}}}
.\label{4.09}
\end{eqnarray}
By using  a proof similar to case (3.8)  of \cite{YY} and (\ref{2.08}),
we have
\begin{eqnarray*}
&&I_{1}\leq
C\left(\prod_{j=1}^{k+1}\|f_{j}\|_{L_{\xi\tau}^{2}}\right)\|g\|_{L_{\xi\tau}^{2}}.
\end{eqnarray*}
When (\ref{4.08})  is valid,  we have that  $|\xi|\sim|\xi_{k+1}|^{-1}$, thus,  we have
\begin{eqnarray}
&&K(\xi_{1},\xi_{2},\cdot\cdot\cdot,\xi_{k},\xi,\tau_{1},\tau_{2},\cdot\cdot\cdot,\tau_{k},\tau)\leq C
\frac{|\xi_{1}||\xi_{2}|^{\frac{1}{4}+\epsilon}|\xi_{k+1}|^{\frac{1}{4}}|\xi|^{1-2\epsilon}}
{\langle\sigma\rangle^{-b_{0}}\prod\limits_{j=1}^{k+1}\langle\sigma_{j}\rangle^{b_{1}}}
.\label{4.010}
\end{eqnarray}
By using  a proof similar to case (3.8)  of \cite{YY},
we have
\begin{eqnarray*}
&&I_{1}\leq  C\|F\|_{X_{0,\frac{1}{2}-\frac{\epsilon}{12}}}\prod\limits_{j=1}^{k+1}\|F_{j}\|_{X_{0,b_{1}}}\leq
C\left(\prod_{j=1}^{k+1}\|f_{j}\|_{L_{\xi\tau}^{2}}\right)\|g\|_{L_{\xi\tau}^{2}}.
\end{eqnarray*}
{\bf Case(3)}: $(\xi_{1},\xi_{2},\cdot\cdot\cdot,\xi_{k},\xi,\tau_{1},\tau_{2},
\cdot\cdot\cdot,\tau_{k},\tau)\in \Omega_{2}$,  we consider
\begin{eqnarray}
&&\left|1-\frac{1}{3\xi_{2}^{2}\xi_{3}^{2}}\right|\geq\frac{1}{2},\label{4.011}\\
&&\left|1-\frac{1}{3\xi_{2}^{2}\xi_{3}^{2}}\right|<\frac{1}{2}.\label{4.012}
\end{eqnarray}
When (\ref{4.011}) is valid,  since $s\geq\frac{1}{2}-\frac{2}{k+1}+2\epsilon$,  we have
\begin{align}
&K(\xi_{1},\xi_{2},\cdot\cdot\cdot,\xi_{k},\xi,\tau_{1},\tau_{2},\cdot\cdot\cdot,\tau_{k},\tau)\leq
C\frac{|\xi_{1}|^{\frac{2}{k+1}+\epsilon}|\xi_{2}^{2}-\xi_{3}^{2}|^{\frac{1}{2}}}
{\langle\sigma\rangle^{-b_{0}}\left(\prod\limits_{j=2}^{k+1}
\langle\xi_{j}\rangle^{s}\right)\left(\prod\limits_{j=1}^{k+1}\langle\sigma_{j}\rangle^{b_{1}}\right)}
\notag\\
&\leq C\frac{|\xi_{2}^{2}-\xi_{3}^{2}|^{\frac{1}{2}}
|\xi_{1}|^{\frac{1}{6}}\prod\limits_{j=4}^{k+1}
\langle\xi_{j}\rangle^{-\frac{ks+\frac{1}{6}-\frac{2}{k+1}-\epsilon}{k-2}}}
{\langle\sigma\rangle^{\frac{1}{2}-\frac{\epsilon}{12}}
\prod\limits_{j=1}^{k+1}\langle\sigma_{j}\rangle^{b}}\nonumber\\
&\leq C\frac{|\phi^{\prime}(\xi_{2})-\phi^{\prime}(\xi_{3})|^{\frac{1}{2}}
|\xi_{1}|^{\frac{1}{6}}\prod\limits_{j=4}^{k+1}
\langle\xi_{j}\rangle^{-\frac{ks+\frac{1}{6}-\frac{2}{k+1}-\epsilon}{k-2}}}{\langle\sigma\rangle^{-b_{0}}
\prod\limits_{j=1}^{k+1}\langle\sigma_{j}\rangle^{b_{1}}}.\label{4.013}
\end{align}
By using (\ref{4.013}) and a proof similar to Case (3) of \cite{YY} and Lemma \ref{Lemma2.2},  we have
\begin{eqnarray*}
&&I_{1}\leq  C\left(\prod_{j=1}^{k+1}\|f_{j}\|_{L_{\xi\tau}^{2}}\right)\|g\|_{L_{\xi\tau}^{2}}.
\end{eqnarray*}
When (\ref{4.012}) is valid,  we have $|\xi_{2}|\sim|\xi_{3}|^{-1}$,
we consider $|\xi|\leq a,|\xi|\geq a$, respectively.

\noindent When $|\xi|\leq a,$  we have
\begin{eqnarray}
&&K(\xi_{1},\xi_{2},\cdot\cdot\cdot,\xi_{k},\xi,\tau_{1},\tau_{2},\cdot\cdot\cdot,\tau_{k},\tau)\leq C
\frac{\prod\limits_{j=1}^{k} \langle \xi_{j}\rangle^{-s}}
{\langle\sigma\rangle^{\frac{1}{2}-\frac{\epsilon}{12}}\prod\limits_{j=1}^{k+1}\langle\sigma_{j}\rangle^{b_{1}}}
.\label{4.014}
\end{eqnarray}
By using  (\ref{4.014}) and a proof similar to (3.15) of \cite{YY},
we have
\begin{eqnarray*}
&&I_{1}\leq C\|F\|_{X_{0,\frac{1}{2}-\frac{\epsilon}{12}}}\prod\limits_{j=1}^{k+1}\|F_{j}\|_{X_{0,b_{1}}}\leq
C\left(\prod_{j=1}^{k+1}\|f_{j}\|_{L_{\xi\tau}^{2}}\right)\|g\|_{L_{\xi\tau}^{2}}.
\end{eqnarray*}
When $|\xi|\geq a,$  we consider
(\ref{4.07}), (\ref{4.08}), respectively.

\noindent When (\ref{4.07}) is valid,  we have
\begin{eqnarray}
&&K(\xi_{1},\xi_{2},\cdot\cdot\cdot,\xi_{k},\xi,\tau_{1},\tau_{2},\cdot\cdot\cdot,\tau_{k},\tau)\leq C
\frac{|\xi_{2}||\xi_{1}|^{-\frac{1}{2}-\epsilon}|\xi_{3}|^{\frac{1}{4}+\epsilon}|\xi|^{1-4\epsilon}}
{\langle\sigma\rangle^{\frac{1}{2}-\frac{\epsilon}{12}}\prod\limits_{j=1}^{k+1}\langle\sigma_{j}\rangle^{b_{1}}}
.\label{4.015}
\end{eqnarray}
By using  (\ref{4.015}) and a proof similar to (3.16) of \cite{YY}, we have
\begin{eqnarray*}
&&I_{1}\leq
C\left(\prod_{j=1}^{k+1}\|f_{j}\|_{L_{\xi\tau}^{2}}\right)\|g\|_{L_{\xi\tau}^{2}}.
\end{eqnarray*}
When (\ref{4.08})  is valid,  we  have  $|\xi|\sim|\xi_{k+1}|^{-1}$. Since $k\geq 6$, and $s>\frac{1}{2}-\frac{2}{k+1}+88\epsilon$, it follows that
\begin{eqnarray}
&&K(\xi_{1},\xi_{2},\cdot\cdot\cdot,\xi_{k},\xi,\tau_{1},\tau_{2},\cdot\cdot\cdot,\tau_{k},\tau)\nonumber\\
&& =
\frac{|\xi_{1}|^{\frac{1}{6}}|\xi_{3}|^{\frac{1}{4}+\epsilon}|\xi_{2}||\xi|^{\frac{2-\epsilon}{12}} |\xi_{1}|^{-\frac{1}{6}}|\xi_{3}|^{-(\frac{1}{4}+\epsilon)}|\xi_{2}|^{-1}|\xi|^{-\frac{2-\epsilon}{12}}|\xi|^{\frac{3}{2}+\epsilon}}
{\langle\sigma\rangle^{\frac{1}{2}-\frac{\epsilon}{12}}\prod\limits_{j=1}^{k+1}\langle\xi_{j}\rangle^{s}\langle\sigma_{j}\rangle^{b_{1}}}
\nonumber\\
&&\leq C\frac{|\xi_{1}|^{\frac{1}{6}}|\xi_{3}|^{\frac{1}{4}+\epsilon}|\xi_{2}||\xi|^{\frac{2-\epsilon}{12}}|\xi_{1}|^{\frac{5+25\epsilon}{12}-2s}}
{\langle\sigma\rangle^{\frac{1}{2}-\frac{\epsilon}{12}}\left(\prod\limits_{j=1}^{k+1}\langle\sigma_{j}
\rangle^{b_{1}}\right)}\nonumber\\
&&\leq C\frac{|\xi_{1}|^{\frac{1}{6}}|\xi_{3}|^{\frac{1}{4}+\epsilon}|\xi_{2}||\xi|^{\frac{2-\epsilon}{12}}}
{\langle\sigma\rangle^{\frac{1}{2}-\frac{\epsilon}{12}}\left(\prod\limits_{j=1}^{k+1}\langle\sigma_{j}
\rangle^{b_{1}}\right)}.\label{4.016}
\end{eqnarray}
By using  (\ref{4.016}), H\"{o}lder inequality, and \eqref{2.05}, \eqref{2.09}-\eqref{2.010}, together with \eqref{2.012}, we have
\begin{eqnarray*}
&&I_{1}\leq C\int_{\xi=\sum\limits_{j=1}^{k+1}\xi_{j}}
\int_{\tau=\sum\limits_{j=1}^{k+1}\tau_{j}}
\frac{g\left(\prod\limits_{j=1}^{k+1}f_{j}\right)|\xi_{1}|^{\frac{1}{6}}|\xi_{3}|^{\frac{1}{4}+\epsilon}|\xi_{2}||\xi|^{\frac{2-\epsilon}{12}}}
{\langle\sigma\rangle^{\frac{1}{2}-\frac{\epsilon}{12}}\left(\prod\limits_{j=1}^{k+1}\langle\sigma_{j}
\rangle^{b_{1}}\right)}d\delta\nonumber\\
&&\leq C\|D_{x}^{\frac{1}{6}}F_{1}\|_{L_{xt}^{6}}\|D_{x}^{\frac{2-\epsilon}{12}}F\|_{L_{xt}^{\frac{6}{1+\epsilon}}}\|D_{x}F_{2}\|_{L_{x}^{\infty}L_{t}^{2}}
\|D_{x}^{\frac{1}{4}+\epsilon}F_{3}\|_{L_{x}^{2}L_{t}^{\infty}}
\left(\prod_{j=4}^{k+1}\|F_{j}
\|_{L_{xt}^{\frac{6(k-2)}{1-\epsilon}}}\right)
\nonumber\\
&&\leq C\|F\|_{X_{0,\frac{1}{2}-\frac{\epsilon}{12}}}
\prod\limits_{j=1}^{k+1}\|F_{j}\|_{X_{0,b_{1}}}\leq
C\left(\prod_{j=1}^{k+1}\|f_{j}\|_{L_{\xi\tau}^{2}}\right)
\|g\|_{L_{\xi\tau}^{2}}.
\end{eqnarray*}
{\bf Case(4)}: $(\xi_{1},\xi_{2},\cdot\cdot\cdot,\xi_{k},\xi,\tau_{1},\tau_{2},
\cdot\cdot\cdot,\tau_{k},\tau)\in \Omega_{3}$, we consider
\begin{eqnarray}
&&\left|1-\frac{1}{3\xi_{3}^{2}\xi_{4}^{2}}\right|\geq\frac{1}{2},\label{4.017}\\
&&\left|1-\frac{1}{3\xi_{3}^{2}\xi_{4}^{2}}\right|<\frac{1}{2},\label{4.018}
\end{eqnarray}
respectively.

\noindent When  (\ref{4.017}) is valid, we have
\begin{eqnarray}
K(\xi_{1},\xi_{2},\cdot\cdot\cdot,\xi_{k},\xi,\tau_{1},\tau_{2},\cdot\cdot\cdot,\tau_{k},\tau)\leq C
\frac{|\xi_{1}|^{\frac{2}{k+1}+\epsilon}|\xi_{3}^{2}-\xi_{4}^{2}|^{\frac{1}{2}}}
{\langle\sigma\rangle^{-b_{0}}\left(\prod\limits_{j=2}^{k+1}
\langle\xi_{j}\rangle^{s}\right)\left(\prod\limits_{j=1}^{k+1}\langle\sigma_{j}\rangle^{b_{1}}\right)}\nonumber\\
\leq C
\frac{|\phi^{\prime}(\xi_{3})-\phi^{\prime}(\xi_{4})|^{\frac{1}{2}}\prod\limits_{j=5}^{k+1}
\langle\xi_{j}\rangle^{-\frac{ks-\frac{2}{k+1}-\epsilon}{k-3}}
}
{\langle\sigma\rangle^{-b_{0}}\langle\sigma_{j}\rangle^{b_{1}}}
.\label{4.019}
\end{eqnarray}
 By using (\ref{4.019}) and Case (4) of \cite{YYDH} and (\ref{2.05}), (\ref{2.07})-(\ref{2.08}) as well as Lemma \ref{Lemma2.2},
we have
\begin{eqnarray*}
&&I_{1}\leq C\left(\prod_{j=1}^{k+1}
\|f_{j}\|_{L_{\xi\tau}^{2}}\right)\|g\|_{L_{\xi\tau}^{2}}.
\end{eqnarray*}
When (\ref{4.018})  is valid,  we have $|\xi_{3}|\sim|\xi_{4}|^{-1}$.

\noindent We consider $|\xi|\leq a,|\xi|\geq a,$ respectively.

\noindent When $|\xi|\leq a,$  we have
\begin{eqnarray}
&&K(\xi_{1},\xi_{2},\cdot\cdot\cdot,\xi_{k},\xi,\tau_{1},\tau_{2},\cdot\cdot\cdot,\tau_{k},\tau)\leq C
\frac{\prod\limits_{j=1}^{k} \langle \xi_{j}\rangle^{-s}}
{\langle\sigma\rangle^{-b_{0}}\prod\limits_{j=1}^{k+1}\langle\sigma_{j}\rangle^{b_{1}}}
.\label{4.020}
\end{eqnarray}
By using  (\ref{4.020}) and a proof similar to (3.21)  of \cite{YY},
we have
\begin{eqnarray*}
&&I_{1}\leq C\|F\|_{X_{0,-b_{0}}}\prod\limits_{j=1}^{k+1}\|F_{j}\|_{X_{0,b_{1}}}\leq
C\left(\prod_{j=1}^{k+1}\|f_{j}\|_{L_{\xi\tau}^{2}}\right)\|g\|_{L_{\xi\tau}^{2}}.
\end{eqnarray*}
When $|\xi|\geq a,$ by using a proof similar to Case $|\xi|\geq a$ of Case (4) of \cite{YY}, we have

\begin{eqnarray*}
&&I_{1}\leq C\|F\|_{X_{0,\frac{1}{2}-\frac{\epsilon}{12}}}\prod\limits_{j=1}^{k+1}\|F_{j}\|_{X_{0,b_{1}}}\leq
C\left(\prod_{j=1}^{k+1}\|f_{j}\|_{L_{\xi\tau}^{2}}\right)\|g\|_{L_{\xi\tau}^{2}}.
\end{eqnarray*}
{\bf Case(5)}: $(\xi_{1},\xi_{2},\cdot\cdot\cdot,\xi_{k},\xi,\tau_{1},\tau_{2},
\cdot\cdot\cdot,\tau_{k},\tau)\in \Omega_{4}$,
this case can be proved similarly to Case (4).\\
\noindent
{\bf Case(6)}: $(\xi_{1},\xi_{2},\cdot\cdot\cdot,\xi_{k},\xi,\tau_{1},\tau_{2},
\cdot\cdot\cdot,\tau_{k},\tau)\in \Omega_{5}$,
 we consider
$|\xi|\leq a,|\xi|\geq a$, respectively.

\noindent When $|\xi|\geq a$.
Obviously,  since  $s\geq\frac{1}{2}-\frac{2}{k+1}+88\epsilon(k\geq 6),$ we have
\begin{eqnarray}
K(\xi_{1},\cdot\cdot\cdot,\xi_{k},\xi,\tau_{1},\cdot\cdot\cdot,\tau_{k},\tau)\leq C
\frac{\langle\xi\rangle^{\frac{2-\epsilon}{12}}\left(\prod\limits_{j=1}^{5}\langle\xi_{j}
\rangle^{\frac{1}{6}-\epsilon}\right)\left(\prod\limits_{j=6}^{k+1}\langle\xi_{j}\rangle^{\frac{61\epsilon}{12(k-4)}-\frac{(k+1)s-\frac{1}{2}-\epsilon}{k-4}}
\right)}
{\langle\sigma\rangle^{-b_{0}}\prod\limits_{j=1}^{k+1}\langle\sigma_{j}\rangle^{b_{1}}}.\label{4.021}
\end{eqnarray}
By using \eqref{4.021}, the H\"older inequality and (\ref{2.05}),  (\ref{2.012})-(\ref{2.013}), we have
\begin{eqnarray*}
&&I_{1}\leq C\|D_{x}^{\frac{2-\epsilon}{12}}F\|_{L_{xt}^{\frac{6}{1+\epsilon}}}
\left(\prod\limits_{j=1}^{5}\|D_{x}^{\frac{1}{6}-\epsilon}F_{j}\|_{L_{xt}^{\frac{12}{2-3\epsilon}}}\right)
\prod\limits_{j=6}^{k+1}\left\|D_{x}^{\frac{61\epsilon}{12(k-4)}-\frac{(k+1)s-\frac{1}{2}-\epsilon}
{k-4}}F_{j}\right\|_{L_{xt}^{\frac{12(k-4)}{13\epsilon}}}
\nonumber\\
&&\leq C\left(\prod_{j=1}^{k+1}\|F_{j}\|_{X_{0,b_{1}}}\right)
\|F\|_{X_{0,-b_{0}}}\leq
C\left(\prod_{j=1}^{k+1}\|f_{j}\|_{L_{\xi\tau}^{2}}\right)
\|g\|_{L_{\xi\tau}^{2}}.
\end{eqnarray*}
When $|\xi|\leq a$.
Obviously,  since  $s\geq\frac{1}{2}-\frac{2}{k+1}+88\epsilon(k\geq6),$ we have
\begin{eqnarray}
K(\xi_{1},\cdot\cdot\cdot,\xi_{k},\xi,\tau_{1},\cdot\cdot\cdot,\tau_{k},\tau)\leq C
\frac{1}
{\langle\sigma\rangle^{-b_{0}}\prod\limits_{j=1}^{k+1}\langle\sigma_{j}\rangle^{b_{1}}}.\label{4.022}
\end{eqnarray}
By using \eqref{4.022}, the H\"older inequality and (\ref{2.05}), we have
\begin{eqnarray*}
&&I_{1}\leq C\|F\|_{L_{xt}^{2}}
\prod\limits_{j=1}^{k+1}\left\|J_{x}^{-s}F_{j}\right\|_{L_{xt}^{2(k+1)}}
\leq C\left(\prod_{j=1}^{k+1}\|F_{j}\|_{X_{0,b_{1}}}\right)
\|F\|_{X_{0,-b_{0}}}\nonumber\\
&&\leq
C\left(\prod_{j=1}^{k+1}\|f_{j}\|_{L_{\xi\tau}^{2}}\right)
\|g\|_{L_{\xi\tau}^{2}}.
\end{eqnarray*}
{\bf Case(7)}: $(\xi_{1},\xi_{2},\cdot\cdot\cdot,\xi_{k},\xi,\tau_{1},
\tau_{2},\cdot\cdot\cdot,\tau_{k},\tau)\in \Omega_{6}$. we consider
$|\xi|\leq a,|\xi|\geq a$, respectively.

\noindent When $|\xi|\geq a$.
Obviously,  since  $s\geq\frac{1}{2}-\frac{2}{k+1}+88\epsilon(k\geq 6),$ we have
\begin{eqnarray}
K(\xi_{1},\cdot\cdot\cdot,\xi_{k},\xi,\tau_{1},\cdot\cdot\cdot,\tau_{k},\tau)\leq C
\frac{\langle\xi\rangle^{\frac{2-\epsilon}{12}}\left(\prod\limits_{j=1}^{5}\langle\xi_{j}
\rangle^{\frac{1}{6}-\epsilon}\right)\left(\prod\limits_{j=6}^{k+1}\langle\xi_{j}\rangle^{\frac{61\epsilon}{12(k-4)}-\frac{(k+1)s-\frac{1}{2}-\epsilon}{k-4}}
\right)}
{\langle\sigma\rangle^{-b_{0}}\prod\limits_{j=1}^{k+1}\langle\sigma_{j}\rangle^{b_{1}}}.\label{4.023}
\end{eqnarray}
By using \eqref{4.023}, the H\"older inequality, and (\ref{2.05}),  (\ref{2.012})-(\ref{2.013}), we have
\begin{eqnarray*}
&&I_{1}\leq C\|D_{x}^{\frac{2-\epsilon}{12}}F\|_{L_{xt}^{\frac{6}{1+\epsilon}}}
\left(\prod\limits_{j=1}^{5}\|D_{x}^{\frac{1}{6}-\epsilon}F_{j}\|_{L_{xt}^{\frac{12}{2-3\epsilon}}}\right)
\prod\limits_{j=6}^{k+1}\left\|D_{x}^{\frac{61\epsilon}{12(k-4)}-\frac{(k+1)s-\frac{1}{2}-\epsilon}
{k-4}}F_{j}\right\|_{L_{xt}^{\frac{12(k-4)}{13\epsilon}}}
\nonumber\\
&&\leq C\left(\prod_{j=1}^{k+1}\|F_{j}\|_{X_{0,b_{1}}}\right)
\|F\|_{X_{0,-b_{0}}}\leq
C\left(\prod_{j=1}^{k+1}\|f_{j}\|_{L_{\xi\tau}^{2}}\right)
\|g\|_{L_{\xi\tau}^{2}}.
\end{eqnarray*}
When $|\xi|\leq a$.
Obviously,  since  $s\geq\frac{1}{2}-\frac{2}{k+1}+88\epsilon(k\geq6),$ we have
\begin{eqnarray}
K(\xi_{1},\cdot\cdot\cdot,\xi_{k},\xi,\tau_{1},\cdot\cdot\cdot,\tau_{k},\tau)\leq C
\frac{1}
{\langle\sigma\rangle^{-b_{0}}\prod\limits_{j=1}^{k+1}\langle\sigma_{j}\rangle^{b_{1}}}.\label{4.024}
\end{eqnarray}
By using \eqref{4.024}, the H\"older inequality, and (\ref{2.05}), we have
\begin{eqnarray*}
&&I_{1}\leq C\|F\|_{L_{xt}^{2}}
\prod\limits_{j=1}^{k+1}\left\|J_{x}^{-s}F_{j}\right\|_{L_{xt}^{2(k+1)}}
\leq C\left(\prod_{j=1}^{k+1}\|F_{j}\|_{X_{0,b_{1}}}\right)
\|F\|_{X_{0,-b_{0}}}\nonumber\\
&&\leq
C\left(\prod_{j=1}^{k+1}\|f_{j}\|_{L_{\xi\tau}^{2}}\right)
\|g\|_{L_{\xi\tau}^{2}}.
\end{eqnarray*}

\noindent{\bf Case(8)}: $(\xi_{1},\xi_{2},\cdot\cdot\cdot,\xi_{k},\xi,\tau_{1},
\tau_{2},\cdot\cdot\cdot,\tau_{k},\tau)\in \Omega_{7}$, we consider
$|\xi|\leq a,|\xi|\geq a$, respectively.

This case can be proved similarly to Case (8) of \cite {YY}.

We completed the proof of Lemma 4.1.

\begin{Remark}
In Lemma \ref{Lemma4.1}, the condition $k\geq6$ is required by Case (3); see \eqref{4.016} for details.
\end{Remark}

\bigskip

\section{ Proof of Theorem 1.2}
\setcounter{equation}{0}
\setcounter{Theorem}{0}

\setcounter{Lemma}{0}

\setcounter{section}{5}
In this section, we use high-low frequency technique to prove Theorem 1.2.

We firstly give the maximal function estimate related to Ostrovsky  equation.
\begin{Lemma}\label{Lemma5.1}
For $f\in H^{\frac{1}{4}}(\R)$, we have
\begin{eqnarray}
&&\left\|U(t)P^{8}f\right\|_{L_{x}^{4}L_{t}^{\infty}}\leq C
\left\|f\right\|_{H^{\frac{1}{4}}(\SR)}\label{5.01}.
\end{eqnarray}
\end{Lemma}

For the proof of Lemma 5.1, we refer the readers to (2.2)  of \cite{GHO}.

 {\noindent} {\bf Proof of Theorem 1.2.} When $f$ is a rapidly decreasing function,
 from Theorem 1.1 of \cite{YZDY}, we know that
Theorem 1.1 is valid. When $f\in H^{\frac{1}{4}}(\R)$, $\forall \epsilon>0,$
by using the density Theorem \cite{Du}, there exists a rapidly decreasing function $g$
 and $h\in H^{s}(\R)$ with
$\|h\|_{H^{s}}<\epsilon$ such that $f=g+h$. For $x\in \R,$ there exists a $y_{0}\in \R$
 such that $x\in [y_{0}-\frac{1}{2}, y_{0}+\frac{1}{2}]:=B_{1}.$  By using the triangle
  inequality, we have
\begin{eqnarray}
&&\lim\limits_{t\longrightarrow0}\left|U(t)f-f\right|\leq \lim\limits_{t\longrightarrow0}
\left|U(t)g-g\right|
+\lim\limits_{t\longrightarrow0}\left|U(t)h-h\right|\leq
\lim\limits_{t\longrightarrow0}\left|U(t)h\right|+|h|\nonumber\\
&&\leq \lim\limits_{t\longrightarrow0}\left|U(t)P^{8}h\right|+\lim\limits_{t\longrightarrow0}
\left|U(t)P_{8}h\right|+|h|
.\label{5.02}
\end{eqnarray}
For any $\alpha>0$ (fixed), we  define
\begin{eqnarray}
&&{\rm E_{\alpha}}=\left\{x\in B_{1}: \lim\limits_{t\longrightarrow0}
\left|U(t)f-f\right|>\alpha\right\}.\label{5.03}
\end{eqnarray}
Combining (\ref{5.02}) with (\ref{5.03}), by using Lemma \ref{Lemma5.1}, (\ref{2.01})  and
(\ref{2.026}) as well as $H^{s}(\R)\hookrightarrow L^{4}(\R)(s\geq\frac{1}{4}),$ we have
\begin{eqnarray}
&&\left|{\rm E_{\alpha}}\right|\leq \left|\left\{x\in B_{1}: \sup\limits_{0<t<1}
\left|P^{8}h\right|>\frac{\alpha}{3}\right\}\right|
\nonumber\\&&\qquad\qquad +
\left|\left\{x\in B_{1}: \sup\limits_{0<t<1}\left|P_{8}\Psi(t)h\right|
>\frac{\alpha}{3}\right\}\right|+\left|\left\{x\in B_{1}:\left|h\right|>
\frac{\alpha}{3}\right\}\right|\nonumber\\
&&\qquad\leq \int_{\left\{x\in B_{1}: \sup\limits_{0<t<1}
\left|P^{8}h\right|>\frac{\alpha}{3}\right\}}\left|\frac{ \sup\limits_{0<t<1}
\left|P^{8}h\right|}{\frac{\alpha}{3}}\right|^{4}dx
+\int_{\left\{x\in B_{1}: \sup\limits_{0<t<1}\left|P_{8}\Psi(t)h\right|
>\frac{\alpha}{3}\right\}}\left|\frac{ \sup\limits_{0<t<1}
\left|P_{8}\Psi(t)h\right|}{\frac{\alpha}{3}}\right|dx\nonumber\\
&&\qquad\qquad+C\|h\|_{L_{x}^{4}}\nonumber\\
&&\leq C\left\|U(t)P^{8}f\right\|_{L_{x}^{4}L_{t}^{\infty}}+C\left|B_{1}\right|
\left\|U(t)P_{8}\Psi(t)f\right\|_{L_{xt}^{\infty}}+\|h\|_{H^{s}}\nonumber\\
&&\leq C\|h\|_{H^{s}}+C\|\Psi(t)U(t)h\|_{X_{0,b}}+C\|h\|_{H^{s}}\nonumber\\
&&\leq C\|h\|_{H^{s}}+C\|h\|_{L^{2}}\leq C\|h\|_{H^{s}}\leq \epsilon.
\end{eqnarray}
Since $\alpha>0$ is arbitrary, we have that $\left|{\rm E_{\alpha}}\right|=0.$

This completes the proof of Theorem  1.2.$\hfill\Box$

\bigskip

\section{Proof of Theorem 1.3}
\setcounter{equation}{0}
\setcounter{Theorem}{0}

\setcounter{Lemma}{0}

\setcounter{section}{6}
In this section, we establish the decomposition  of  the solution to
(\ref{1.01})-(\ref{1.02}) which is just Lemma 6.1 and show
 $\lim\limits_{t\longrightarrow0}\Psi(t)U(t)f_{N}=f_{N}$,
which is just Lemma 6.2.

To prove Theorem 1.3, we firstly prove Lemma 6.1.
\begin{Lemma}\label{Lemma6.1}(The decomposition  of  the solution to (\ref{1.01})-(\ref{1.02})).
Let $k\geq5$, $f\in H^{s_{1}}(\R)(s_{1}>\max\{\frac{1}{2}-\frac{2}{k},\frac{1}{4}\})$ and  $u$ be the solution to
(\ref{1.01})-(\ref{1.02}). Then,  $\forall \epsilon_{1}>0$ and sufficiently large  $N>0,$  we have
\begin{eqnarray*}
u=\Psi(t)U(t)f-\frac{\Psi_{\delta}(t)}{k+1}\int_{0}^{t}U(t-\tau)\partial_{x}(u^{k+1})
d\tau =F_{1N}(x,t)+F_{2}^{N}(x,t),t\in[-\delta,\delta].
\end{eqnarray*}
Here
\begin{eqnarray*}
&&(i):\Psi(t)U(t)f_{N}\in C([-\delta,\delta];C^{\infty}(\R)),\\
&&(ii):F_{1N}(x,t)=\Psi(t)U(t)f_{N}-\frac{\Psi_{\delta}(t)}{k+1}
\int_{0}^{t}U(t-\tau)P_{N}\partial_{x}(u_{N}^{k+1})d\tau\in C([-\delta,\delta];C^{\infty}(\R)),\\
&&(iii):F_{2}^{N}(x,t)=\Psi(t)U(t)f^{N}\nonumber\\&&\qquad\qquad
-\frac{\Psi_{\delta}(t)}{k+1}\left[\int_{0}^{t}U(t-\tau)\partial_{x}(u^{k+1})d\tau
-\int_{0}^{t}U(t-\tau)P_{N}\partial_{x}(u_{N}^{k+1})d\tau\right],
\end{eqnarray*}
where
\begin{eqnarray*}
&&u_{N}=\frac{1}{\sqrt{2\pi}}\int_{|\xi|\leq N}e^{ix\xi}\mathscr{F}_{x}u(\xi,t)d\xi,\\
&&f^{N}=\frac{1}{\sqrt{2\pi}}\int_{|\xi|\geq N}e^{ix\xi}\mathscr{F}_{x}f(\xi)d\xi,\\
&&\|F_{2}^{N}\|_{L_{x}^{4}L_{t}^{\infty}}<C\|F_{2}^{N}\|_{X_{s_{1},b}}<C\epsilon_{1}.
\end{eqnarray*}
\end{Lemma}
\noindent{\bf Proof.} We define $b=\frac{1}{2}+\frac{\epsilon}{24}$
 and $b^{\prime}=-\frac{1}{2}+\frac{\epsilon}{12}$. By using Lemma 3.1 of \cite{YY}, we have
\begin{eqnarray}
&&\left\|\partial_{x}(u_{N}^{k+1})\right\|_{X_{s_{1},b^{\prime}}}
\leq C\|u_{N}\|_{X_{s_{1},b}}^{k+1}\leq \left( 2C\|f_{N}\|_{H^{s_{1}}}\right)^{k+1}\label{6.01},\\
&&\left\|\partial_{x}(u^{k+1}-u_{N}^{k+1})\right\|_{X_{s_{1},b^{\prime}}}\nonumber\\&&
\leq C\|u^{N}\|_{X_{s_{1},b}}\left[\|u_{N}\|_{X_{s_{1},b}}+
\|u\|_{X_{s_{1},b}}\right]^{k}\leq (4C\|f\|_{H^{s_{1}}})^{k}\|f^{N}\|_{H^{s_{1}}}\label{6.02}.
\end{eqnarray}
Since $f\in H^{s_{1}}(\R)$ and $u\in X_{s_{1},b}$, $\forall \epsilon_{1}>0$
 and sufficiently large  $N>0,$ we have
\begin{eqnarray}
&&\|f^{N}\|_{H^{s_{1}}}\leq\epsilon,\|u^{N}\|_{X_{s_{1},b}}\leq C\epsilon_{1}\label{6.03},\\
&&\left\|\partial_{x}P^{N}(u_{N}^{k+1})\right\|_{X_{s_{1},b^{\prime}}}\leq C\epsilon_{1}\label{6.04},\\
&&\left\|\partial_{x}(u^{k+1}-u_{N}^{k+1})\right\|_{X_{s_{1},b^{\prime}}}\leq C\epsilon_{1}\label{6.05}.
\end{eqnarray}
To prove that $F_{1N}\in C([-\delta,\delta];C^{\infty}(\R)),$
 it suffices to  prove
\begin{eqnarray}
F_{1N}\in C([-\delta,\delta];C^{m}(\R))\label{6.06}
\end{eqnarray}
for arbitrary  $m\in \mathbf{N}$.

\noindent From \cite{YY}, we know that (\ref{1.01})-(\ref{1.02})
possess a unique solution  $u$ on $[-\delta,\delta]$ which can be rewritten
as follows:
\begin{eqnarray}
&&u=\Psi(t)U(t)f-\frac{1}{k+1}\Psi_{\delta}(t)\int_{0}^{t}U(t-\tau)\partial_{x}(u^{k+1})d\tau \label{6.07}
\end{eqnarray}
and
\begin{eqnarray}
&&\|u\|_{X_{s_{1},b}}\leq 2C\|f\|_{H^{s_{1}}}.\label{6.08}
\end{eqnarray}
We define $F_{1N}(x,t)=\Psi(t)U(t)f_{N}-\frac{1}{k+1}\Psi_{\delta}(t)\int_{0}^{t}U(t)(t-\tau)P_{N}(u_{N}^{k+1})d\tau$.
We claim that $F_{1N}(x,t)\in X_{s,b}(s>{\rm max}\left\{m+1,s_{1}\right\})$.
Applying Lemma 3.1 of \cite{YY}, and (\ref{2.01})-(\ref{2.02}), 
\eqref{6.01} to $F_{1N}$ yields
\begin{eqnarray}
\left\|F_{1N}\right\|_{X_{s,b}}&&\leq \left\|\Psi(t)U(t)f_{N}\right\|_{X_{s,b}}+
\left\|\Psi_{\delta}(t)\int_{0}^{t}U(t-\tau)P_{N}\partial_{x}(u_{N}^{k+1})d\tau\right\|_{X_{s,b}}
\nonumber\\&&\leq
 C\|f_{N}\|_{H^{s}}+C\delta^{\frac{\epsilon}{12}}\left\|P_{N}\partial_{x}(u_{N}^{k+1})\right\|_{X_{s,b^{\prime}}}\nonumber\\
&&\leq CN^{s}\|f\|_{L^{2}}+CN^{s-s_{1}}\delta^{\frac{\epsilon}{12}}
\left\|\partial_{x}(u_{N}^{k+1})\right\|_{X_{s_{1},b^{\prime}}}\nonumber\\
&&\leq CN^{s}\|f\|_{L^{2}}+CN^{s-s_{1}}\delta^{\frac{\epsilon}{12}}
\left\|u_{N}\right\|_{X_{s_{1},b}}^{k+1}\nonumber\\
&&\leq CN^{s}\|f\|_{L^{2}}+CN^{s-s_{1}}\delta^{\frac{\epsilon}{12}}
\left\|u\right\|_{X_{s_{1},b}}^{k+1}\nonumber\\
&&\leq CN^{s}\|f\|_{L^{2}}+CN^{s-s_{1}}\delta^{\frac{\epsilon}{12}}
(2C\|f\|_{H^{s_{1}}})^{k+1}.\label{6.09}
\end{eqnarray}
Combining (\ref{6.09}) and Lemma \ref{Lemma3.6}, we have
\begin{eqnarray}
&&F_{1N}\in C([-\delta,\delta];H^{s}(\R)).\label{6.010}
\end{eqnarray}
From (\ref{6.010}), we have
\begin{eqnarray}
&&F_{1N}\in C([-\delta,\delta];C^{m}(\R)).\label{6.011}
\end{eqnarray}
We define $F_{2}^{N}(x,t)=I_{1}+I_{2}+I_{3}.$
Here
\begin{eqnarray}
&&I_{1}=\Psi(t)U(t)f^{N},\label{6.012}\\
&&I_{2}=-\frac{1}{k+1}\Psi_{\delta}(t)
\int_{0}^{t}U(t-\tau)P^{N}\partial_{x}(u_{N}^{k+1})d\tau,\label{6.013}\\
&&I_{3}=-\frac{1}{k+1}\Psi_{\delta}(t)
\int_{0}^{t}U(t-\tau)\partial_{x}(u^{k+1}-(u_{N})^{k+1})d\tau\label{6.014}.
\end{eqnarray}
From (\ref{2.01}) and (\ref{6.01})-(\ref{6.05}),  we have
\begin{eqnarray}
&&\|I_{1}\|_{X_{s_{1},b}}=\left\|\Psi(t)U(t)f^{N}\right\|_{X_{s_{1},b}}
\leq C\|f^{N}\|_{H^{s}}\leq C\epsilon_{1}\label{6.015},\\
&&\left\|I_{2}\right\|_{X_{s_{1},b}}\leq C\delta^{\frac{\epsilon}{12}}
\left\|P^{N}\partial_{x}(u_{N}^{k+1})\right\|_{X_{s_{1},b^{\prime}}}
\leq C\delta^{\frac{\epsilon}{12}}\epsilon_{1}\label{6.016},\\
&&\left\|I_{3}\right\|_{X_{s_{1},b}}\leq C\delta^{\frac{\epsilon}{12}}
\left\|\partial_{x}\left[u^{k+1}-(u^{N})^{k+1}\right]
\right\|_{X_{s_{1},b^{\prime}}}\nonumber\\
&&\qquad\qquad\leq C\delta^{\frac{\epsilon}{12}}\|u^{N}\|_{X_{s_{1},b}}
\left[\|u_{N}\|_{X_{s_{1},b}}+\|u\|_{X_{s_{1},b}}\right]^{k}\nonumber\\
&&\qquad\qquad\leq
 C\delta^{\frac{\epsilon}{12}}(4C\|f\|_{H^{s_{1}}})^{k}\epsilon_{1}\label{6.017}.
\end{eqnarray}

This ends the proof of Lemma 6.1.

\begin{Lemma}\label{Lemma6.2}
Let $f\in L^{2}(\R)$ and $N>0.$
Then, we have
\begin{eqnarray}
\lim\limits_{t\longrightarrow0}U(t)f_{N}=f_{N}.\label{6.018}
\end{eqnarray}
\end{Lemma}
\noindent{\bf Proof.} $\forall \epsilon_{1}>0,$ there exists
$\delta_{1}<\left(\frac{\epsilon_{1}}{4\|f\|_{L^{2}}}\right)^{2}$
 such that $4\delta_{1}^{\frac{1}{2}}\|f\|_{L^{2}}<\epsilon_{1}.$
By using a direct computation, we have
\begin{eqnarray}
\left|U(t)f_{N}-f_{N}\right|=C\left|\int_{|\xi|\leq N}e^{ix\xi}
\left[e^{it(\xi^{3}-\frac{1}{\xi})}-1\right]\mathscr{F}_{x}f(\xi)d\xi\right|\leq C(I_{1}+I_{2}),\label{6.019}
\end{eqnarray}
where
\begin{eqnarray*}
&&I_{1}=\left|\int_{|\xi|\leq \delta_{1}}e^{ix\xi}
\left[e^{it(\xi^{3}-\frac{1}{\xi})}-1\right]\mathscr{F}_{x}f(\xi)d\xi\right|,\nonumber\\&&
I_{2}=\left|\int_{\delta_{1}<|\xi|\leq N}e^{ix\xi}
\left[e^{it(\xi^{3}-\frac{1}{\xi})}-1\right]\mathscr{F}_{x}f(\xi)d\xi\right|.
\end{eqnarray*}
By using the Cauchy-Schwarz inequality and
 $\left|e^{it(\xi^{3}-\frac{1}{\xi})}-1\right|\leq2$, we have
\begin{eqnarray}
&&I_{1}=\left|\int_{|\xi|\leq \delta_{1}}e^{ix\xi}
\left[e^{it(\xi^{3}-\frac{1}{\xi})}-1\right]\mathscr{F}_{x}f(\xi)d\xi\right|\nonumber\\&&
\quad\leq 2\int_{|\xi|\leq \delta_{1}}|\mathscr{F}_{x}f(\xi)|d\xi\leq 2\delta_{1}^{\frac{1}{2}}\left[\int_{|\xi|\leq \delta_{1}}|\mathscr{F}_{x}f(\xi)|^{2}d\xi\right]^{\frac{1}{2}}<\frac{\epsilon_{1}}{2}.\label{6.020}
\end{eqnarray}
By using the Cauchy-Schwarz inequality and
$$\left|e^{it(\xi^{3}-\frac{1}{\xi})}-1\right|\leq
|t|\left[|\xi|^{3}+\frac{1}{|\xi|}\right]\leq |t|\left[N^{3}+\frac{1}{\delta_{1}}\right],$$
when $|t|\leq\frac{2\delta_{1}^{\frac{3}{2}}}{N^{\frac{1}{2}}(N^{3}\delta_{1}+1)}$,
we have
\begin{eqnarray}
&&I_{2}=\left|\int_{\delta_{1}<|\xi|\leq N}e^{ix\xi}
\left[e^{it\xi^{2}}-1\right]\mathscr{F}_{x}f(\xi)d\xi\right|\nonumber\\&&
\leq |t|\left[N^{3}+\frac{1}{\delta_{1}}\right]\int_{\delta_{1}<|\xi|\leq N}|\mathscr{F}_{x}f(\xi)|d\xi\nonumber\\
&&\leq |t|\left[N^{3}+\frac{1}{\delta_{1}}\right]N^{\frac{1}{2}}
\left[\int_{\delta_{1}<|\xi|\leq N}|\mathscr{F}_{x}f(\xi)|^{2}d\xi\right]^{\frac{1}{2}}\nonumber\\
&&\leq 2\delta_{1}^{\frac{1}{2}}\|f\|_{L^{2}}<\frac{\epsilon_{1}}{2}.\label{6.021}
\end{eqnarray}
Combining (\ref{6.019}) with (\ref{6.020})-(\ref{6.021}), we have
\begin{eqnarray}
\left|U(t)f_{N}-f_{N}\right|<\epsilon_{1}\label{6.022}.
\end{eqnarray}

This ends the proof of Lemma 6.2.

\begin{Lemma}\label{Lemma6.3}
Let $f\in L^{2}(\R)$.
Then, we have
\begin{eqnarray}
\lim\limits_{t\longrightarrow0}\Psi(t)U(t)f_{N}=f_{N}.\label{6.023}
\end{eqnarray}
\end{Lemma}
\noindent{\bf Proof.} By using the triangle inequality and Lemma \ref{Lemma6.2}, we have
\begin{eqnarray}
&&\lim\limits_{t\longrightarrow0}\left|\Psi(t)U(t)f_{N}-f_{N}\right|
\leq \lim\limits_{t\longrightarrow0}\left|(\Psi(t)-1)U(t)f_{N}\right|+
\lim\limits_{t\longrightarrow0}\left|U(t)f_{N}-f_{N}\right|\nonumber\\
&&\leq \lim\limits_{t\longrightarrow0}\left|(\Psi(t)-1)U(t)f_{N}\right|
\nonumber\\
&&\leq\lim\limits_{t\longrightarrow0}\left|(\Psi(t)-1)\right|
\left\|U(t)f_{N}\right\|_{H^{1}}\nonumber\\
&&\leq CN\lim\limits_{t\longrightarrow0}\left|(\Psi(t)-1)\right|\|f\|_{L^{2}}
=0.
\end{eqnarray}

This ends the proof of Lemma 6.3.

Now we use Lemmas 6.1-6.3  to prove Theorem 1.3.

\noindent{\bf Proof of Theorem 1.3.}

Obviously, by using Lemmas \ref{Lemma6.1}, \ref{Lemma6.3}, we have
\begin{eqnarray}
&&\lim_{t\rightarrow0}|u-f|\leq \lim_{t\rightarrow0}
|F_{1N}-f_{N}|+\lim_{t\rightarrow0}|F_{2}^{N}-f^{N}|\nonumber\\
&&\leq \lim_{t\rightarrow0}|F_{1N}-\Psi(t)U(t)f_{N}|+\lim_{t\rightarrow0}
|\Psi(t)U(t)f_{N}-f_{N}|+|F_{2}^{N}-f^{N}|\nonumber\\
&&\leq \lim_{t\rightarrow0}|F_{2}^{N}|+|f^{N}|.\label{6.024}
\end{eqnarray}
For $\lambda>0$, by using the Chebyshev inequality and Sobolev embedding
 $H^{s_{1}}(\R)\hookrightarrow L^{4}(\R)(s_{1}\geq\frac{1}{4})$ as well as Lemma \ref{Lemma6.1}, we have
\begin{eqnarray}
&&\left|\left\{x\in \R:\lim\limits_{t\rightarrow0}|u-f|>\lambda\right\}\right|\leq \left|\left\{x\in \R:\sup_{0<t<\delta}|F_{2}^{N}|>\frac{\lambda}{2}\right\}\right|\nonumber\\
&&+\left|\left\{x\in \R:|f^{N}|>\frac{\lambda}{2}\right\}\right|
\leq C\lambda^{-4}\left\|\sup_{0< t<\delta}|F_{2}^{N}|\right\|_{L_{x}^{4}}^{4}
+C\lambda^{-4}\left\|f^{N}\right\|_{L_{x}^{4}}^{4}\nonumber\\
&&\leq C\lambda^{-4}\left\|F_{2}^{N}\right\|_{L_{x}^{4}L_{t}^{\infty}}^{4}
+C\lambda^{-4}\left\|P^{N}f\right\|_{H^{s_{1}}}^{4}\nonumber\\
&&\leq C\lambda^{-4}\left\|F_{2}^{N}\right\|_{X_{s_{1},b}}^{4}
+C\lambda^{-4}\left\|P^{N}f\right\|_{H^{s_{1}}}^{4}\nonumber\\
&&\leq C\lambda^{-4}\epsilon_{1}^{4}.\label{6.025}
\end{eqnarray}
From (\ref{6.025}), we have
\begin{eqnarray}
\left|\left\{x\in \R:\lim\limits_{t\rightarrow0}|u-f|>\lambda\right\}\right|=0.
\end{eqnarray}
Since $\lambda$  is arbitrary, we have
\begin{eqnarray}
\lim\limits_{t\longrightarrow 0}u(x,t)=f(x)\label{6.026}.
\end{eqnarray}

This completes the proof of Theorem  1.3.$\hfill\Box$

\bigskip

\section{Proof of Theorem 1.4}
\setcounter{equation}{0}
\setcounter{Theorem}{0}

\setcounter{Lemma}{0}

\setcounter{section}{7}

In this section,  by using the idea of \cite{Compaan},  we will prove Theorem 1.4.

\noindent{\bf Proof of Theorem 1.4.}
By using Duhamel's formula, we can rewrite  the local solution to (\ref{1.01})-(\ref{1.02})  as follows.
\begin{eqnarray}
u=U(t)f+\frac{1}{k+1}\int_{0}^{t}U(t-\tau)\partial_{x}(u^{k+1})d\tau,\label{7.01}
\end{eqnarray}
from Theorem 1.1 of \cite{YY}, we know that when 
$f\in H^{s}(\R)(s\geq \frac{1}{2}-\frac{2}{k+1}+88\epsilon)(k\geq6)$,
(\ref{7.01}) possesses a unique solution.
For $t\in [-\delta,\delta]$, from Lemmas \ref{Lemma2.1} and \ref{Lemma4.1}, we have
$u_{1}:=U(t)f\in H^{s}(\R)$
 and
 $u_{2}:=\frac{1}{k+1}\int_{0}^{t}U(t-\tau)\partial_{x}(u^{k+1})
 d\tau\in X_{\frac{1}{2}+\epsilon,b_{1}}\subset
  C([-\delta,\delta]; H^{\frac{1}{2}+\epsilon}(\R))$ with the aid of Lemma \ref{Lemma3.6}.

This completes the proof of Theorem  1.4.$\hfill\Box$

\bigskip

\section{Proof of Theorem 1.5}
\setcounter{equation}{0}
\setcounter{Theorem}{0}

\setcounter{Lemma}{0}

\setcounter{section}{8}

In this section,  by using the idea of \cite{Compaan},  we will prove Theorem 1.5.

\noindent{\bf Proof of Theorem 1.5.} We define
$b:=\frac{1}{2}+\frac{\epsilon}{24},b_{1}=\frac{1}{2}-\frac{\epsilon}{12}.$
By using Duhamel's formula, we can rewrite  the local solution to (\ref{1.01})-(\ref{1.02})  as follows.
\begin{eqnarray}
u=\Psi(t)U(t)f+\frac{1}{k+1}\Psi\left(\frac{t}{\delta}\right)
\int_{0}^{t}U(t-\tau)\partial_{x}(u^{k+1})d\tau.\label{8.01}
\end{eqnarray}
When $t\in [-\delta,\delta]$, we have
\begin{eqnarray}
u=U(t)f+\frac{1}{k+1}\int_{0}^{t}U(t-\tau)\partial_{x}(u^{k+1})d\tau.\label{8.02}
\end{eqnarray}
We define
\begin{eqnarray}
V(t)=u-U(t)f, F(t)=\|V(t)\|_{L_{x}^{\infty}}.\label{8.03}
\end{eqnarray}
For $k\geq6$, and $s\geq \frac{1}{2}-\frac{2}{k+1}+88\epsilon,$   
by using Theorem \ref{Theorem4}, Lemma \ref{Lemma3.7} and
 $H^{\frac{1}{2}+\epsilon}\hookrightarrow L^{\infty}$,  since  $V(t) \in X_{\frac{1}{2}+\epsilon,b}(b>1/2),$ we have
\begin{eqnarray}
&&|F(t_{1})-F(t_{2})|=|\|V(t_{1})\|_{L_{x}^{\infty}}-\|V(t_{2})\|_{L_{x}^{\infty}}|
\leq \|V(t_{1})-V(t_{2})\|_{L_{x}^{\infty}}\nonumber\\
&&\leq C\|V(t_{1})-V(t_{2})\|_{H^{\frac{1}{2}+\epsilon}}\longrightarrow0
\end{eqnarray}
as $t_{1}\longrightarrow t_{2}.$

This completes the proof of Theorem  1.5.$\hfill\Box$

\bigskip

\section{Proof of Theorem 1.6}
\setcounter{equation}{0}
\setcounter{Theorem}{0}

\setcounter{Lemma}{0}

\setcounter{section}{9}
This section is devoted to establishing Theorem 1.6 with the aid of \cite{YLHHY}.

\begin{Lemma}\label{Lemma9.1}
(Riemann-Lebesgue) Let $g\in L^{1}(\R)$.
Then, we have
\begin{eqnarray}
\lim\limits_{|x|\longrightarrow+\infty}\mathscr{F}_{x}g=0.\label{9.01}
\end{eqnarray}
\end{Lemma}

Lemma 9.1 can be found in Theorem 1.2 in \cite{SW}.

\begin{Lemma}\label{Lemma9.2}
Let $g\in H^{s}(\R^{n})(s>\frac{n}{2})$.
Then, we have
\begin{eqnarray}
\lim\limits_{|x|\longrightarrow+\infty}g=0.\label{9.02}
\end{eqnarray}
\end{Lemma}
\noindent {\bf Proof.} Since $g\in H^{s}(\R^{n})(s>\frac{n}{2})$, it follows that
\begin{eqnarray}
\mathscr{F}_{x}g\in L^{1}(\R^{n}).\label{9.03}
\end{eqnarray}
Combining (\ref{9.02}) with (\ref{9.01}), we have that
\begin{eqnarray}
\lim\limits_{|x|\longrightarrow+\infty}g(-x)=\lim\limits_{|x|\longrightarrow+\infty}g(x)=0.\label{9.04}
\end{eqnarray}

This completes the proof of Lemma 9.2.

Now we are in a position to prove Theorem 1.6.
When $f\in H^{s}(\R)(s>\frac{1}{2}-\frac{2}{k+1})(k\geq6),$ by using Lemmas 
\ref{Lemma4.1}, \ref{Lemma3.6}, we have that
\begin{eqnarray}
&&u_{2}=:\frac{1}{k+1}\int_{0}^{t}U(t-\tau)
\partial_{x}(u^{k+1})d\tau \in X_{\frac{1}{2}+\epsilon,\>b_{1}}
\subset C([-\delta,\delta]);H^{\frac{1}{2}+\epsilon}(\R))
\label{9.05}.
\end{eqnarray}
From (\ref{9.05}), we know that
\begin{eqnarray}
\sup\limits_{t\in [-\delta,\delta]}\|u_{2}\|_{H^{\frac{1}{2}+\epsilon}}<+\infty.\label{9.06}
\end{eqnarray}
From (\ref{9.06}), for $\forall t\in[-\delta,\delta],$ we have that
\begin{eqnarray}
u_{2}\in H^{\frac{1}{2}+\epsilon}.\label{9.07}
\end{eqnarray}
Combining  (\ref{9.07}) with Lemma \ref{Lemma9.2}, for $\forall t\in[-\delta,\delta],$ we have that
\begin{eqnarray}
\lim\limits_{|x|\longrightarrow+\infty}u_{2}=0.\label{9.08}
\end{eqnarray}
When $f\in \hat{L}^{\infty}(\R)$,
since
\begin{eqnarray}
\|\mathscr{F}_{x} (U(t)f)\|_{L^{1}}=\|\mathscr{F}_{x}f\|_{L^{1}}=\|f\|_{\hat{L}^{\infty}(\SR)}.\label{9.09}
\end{eqnarray}
From (\ref{9.010}), we have that for $\forall t\in [-\delta,\delta],$ we have
\begin{eqnarray}
U(t)f \in \hat{L}^{\infty}(\R).\label{9.010}
\end{eqnarray}
Combining Lemma \ref{Lemma9.1} with (\ref{9.010}), we have that
\begin{eqnarray}
\lim\limits_{|x|\rightarrow \infty}U(t)f=0.\label{9.011}
\end{eqnarray}
Combining  (\ref{9.08}) with  (\ref{9.011}),  we have that
\begin{eqnarray*}
\lim\limits_{|x|\longrightarrow+\infty}u=0.\label{9.012}
\end{eqnarray*}

We have finished the proof of Theorem 1.6.

\bigskip
\bigskip

\leftline{\large \bf Acknowledgments}
Wei Yan was supported by Natural Science Foundation of Henan province (No. 262300421062)  and Meihua Yang was
 supported by NSFC grants (No. 11971184).

\end{document}